\documentclass[final,leqno]{siamltex704}
\usepackage{amsmath}
\usepackage{graphicx}
\usepackage[notcite,notref]{showkeys}
\usepackage{mathrsfs}
\usepackage{appendix}
\usepackage{float}
\usepackage{amsfonts,amssymb}
\usepackage{dsfont}
\usepackage{pifont}
\usepackage{wrapfig} 
\usepackage{hyperref}
\usepackage{multirow}
\numberwithin{equation}{section}
\def\3bar{{|\hspace{-.02in}|\hspace{-.02in}|}}

\def\T{{\mathcal{T}}}

\def\pT{{\partial T}}

\def\W{{\mathcal{W}}}
\def\CL{{\mathcal{L}}}
\def\CQ{{\mathcal{Q}}}

\def\bv{{\mathbf{v}}}
\def\bn{{\mathbf{n}}}

\def\action#1#2{(\!(#1, #2)\!)_T}
\def\bmu{{\boldsymbol{\mu}}}
\def\brho{{\boldsymbol{\rho}}}
\def\bsigma{{\boldsymbol{\sigma}}}

\newtheorem{algorithm}{Primal-Dual Weak Galerkin Algorithm}[section]

\title {A Primal-Dual Weak Galerkin Finite Element Method for
 Fokker-Planck Type Equations}

 \author{
Chunmei Wang\thanks{ Department of Mathematics, Texas State
University, San Marcos, TX 78666, USA.
The research of Chunmei Wang was partially supported by National
Science Foundation Award DMS-1522586 and DMS-1648171.
} \and Junping Wang\thanks{Division of Mathematical Sciences,
National Science Foundation, Arlington, VA 22230 (jwang@nsf.gov).
The research of Junping Wang was supported by the NSF IR/D
program, while working at National Science Foundation. However,
any opinion, finding, and conclusions or recommendations expressed
in this material are those of the author and do not necessarily
reflect the views of the National Science Foundation.}}

\begin{document}

\maketitle

\begin{abstract}
This paper presents a primal-dual weak Galerkin (PD-WG) finite
element method for a class of second order elliptic equations of
Fokker-Planck type. The method is based on a variational form
where all the derivatives are applied to the test functions so
that no regularity is necessary for the exact solution of the
model equation. The numerical scheme is designed by using locally
constructed weak second order partial derivatives and the weak
gradient commonly used in the weak Galerkin context. Optimal order
of convergence is derived for the resulting numerical solutions.
Numerical results are reported to demonstrate the performance of
the numerical scheme.
\end{abstract}



\pagestyle{myheadings}

\section{Introduction} The Fokker-Planck equation
plays a critical role in statistical physics and in the study of
fluctuations in physical and biological systems
\cite{fokker,planck,Risken,stratonovich,gardiner}. In statistical
physics, it is a second order partial differential equation that
describes the time evolution of the probability density function
of the velocity of a particle under the influence of drag forces
and random forces resulting from Gaussian white noise. The general
setting of the Fokker-Planck equation is as follows. Given an open
domain $\Omega\subset{\mathbb R}^d$ (the $d$-dimensional Euclidean
space) and a terminal time $T$, we seek for a time-dependent
density function $p=p(x,t): \ \Omega\times [0,T]\to {\mathbb R}$
satisfying
\begin{equation}\label{EQ:FPE}
\begin{split}
\partial_t p +\nabla\cdot (\bmu p)-\frac{1}{2}\sum_{i,j=1}^d
\partial_{ij}^2(a_{ij}p) & = \ 0,\qquad
t\in (0, T), \ x\in \Omega,\\
p(x, 0) & = \ p_0(x),\qquad x\in \Omega,
\end{split}
\end{equation}
where $\partial_{ij}^2=\frac{\partial}{\partial
x_j}\frac{\partial}{\partial x_i}$ is the second order partial
derivative in the directions $x_i$ and $x_j$,
$a(x)=\{a_{ij}(x)\}_{d\times d}$ is the diffusion tensor,
$\bmu=(\mu_1,\cdots, \mu_d)$ is the drift vector, and $p_0=p_0(x)$
is the initial profile of the density function. Two common
boundary conditions for (\ref{EQ:FPE}) can be imposed: Dirichlet
for the density function and Neumann condition for the flux. A
homogeneous Dirichlet boundary data corresponds to the case where
particles exit once they reach the boundary, and a prescribed flow
or Neumann boundary condition represents a known current of
particles crossing the boundary in the normal direction.

Numerical methods for the Fokker-Planck equation have several
challenges involving various difficulties of different nature.
Among them are the high dimensionality and lack of solution
regularity for the probability density function. For example, in
classical statistical mechanics, the Fokker-Planck equation
characterizes a joint probability density function in many phase
variables so that the dimension $d$ might be a big number to deal
with. It can also be seen that for non-smooth diffusion tensor
$a(x)$, the resulting probability density function $p=p(x)$
exhibits a shock-like discontinuity that needs to be resolved
numerically. In addition, some conservation properties such as
mass conservation and solution non-negativity property must be
retained by the numerical solutions.

Various finite element methods have been designed for the
Fokker-Planck equation for a numerical computation of the
probability density function, see
\cite{bhandari-sherrer,langley,langtangen,bergman,spencer-bergman,masud,kumar-narayana}
and the references cited therein. In
\cite{bhandari-sherrer,langley}, the stationary Fokker-Planck
equation was discretized by using Galerkin finite element methods
based on a weak form obtained from the usual integration by parts.
In \cite{langtangen}, another Galerkin finite element method was
used to solve the Fokker-Planck equation in combination with a
generalized Lagrange multiplier method to handle the associated
integral constraint. In \cite{bergman}, the authors applied the
usual $C^0$ finite element method to the Fokker-Planck system
subject to both additive and multiplicative white noise
excitations. In \cite{masud}, the authors developed a framework
for multi-scale finite element methods for the solution of the
multi-dimensional Fokker-Planck equation in stochastic structural
dynamics. All the aforementioned finite element methods assumed
smooth or constant diffusion tensor $a(x)$ for the Fokker-Planck
equation so that a regular weak form can be derived for the
system. In \cite{kumar-narayana}, a $C^0$ finite element scheme
was described and numerically tested for the transient
Fokker-Planck equation without any smoothness assumption on
$a(x)$; but no theory of convergence was developed for the
numerical method.

For a smooth diffusion tensor, the second order differential part
of the Fokker-Plank equation can be reformulated as
\begin{equation}\label{EQ:Reformulation}
\frac12 \partial_{ij}^2(a_{ij}p) = \frac12 \partial_j (a_{ij}(x)
\partial_i p) + \frac12
\partial_j ( (\partial_i a_{ij}) p).
\end{equation}
Therefore, the equation (\ref{EQ:FPE}) can be viewed as a
time-evolving convection-diffusion equation in divergence form.
Aside from the high dimensionality issue, the corresponding
Fokker-Planck equation is then considered as a relatively less
challenging problem to solve numerically. But for non-smooth
diffusion tensor, the formulation (\ref{EQ:Reformulation}) no
longer holds true, and the exact profile of the density function
$p=p(x)$ possesses discontinuities that are not known a priori so
that the existing finite element methods have difficulty to apply.
The goal of this paper is to develop a new finite element method
that addresses the numerical challenges arising from the
non-smoothness nature of the diffusion tensor $a(x)$ in the
Fokker-Planck equation.

For simplicity, we consider a Fokker-Planck type model equation
with homogeneous Dirichlet boundary condition. The model problem
seeks an unknown function $u=u(x)$ satisfying
\begin{equation}\label{1}
\begin{split}
\nabla \cdot (\bmu u)-\frac{1}{2}\sum_{i,j=1}^d
\partial^2_{ij}(a_{ij}u)=&f,\quad \text{in}\
\Omega,\\
u =& 0,\quad \text{on}\ \partial\Omega,
\end{split}
\end{equation}
where $\Omega$ is an open bounded domain in $\mathbb R^d$ with
Lipschitz continuous boundary $\partial\Omega$ and $f\in
L^2(\Omega)$ is a given function. Two of our main motivations for
this selection of the model problem are: (1) a complete
understanding of the Fokker-Planck equation strongly depends on
the numerical properties for (\ref{1}) as it offers a projection
operator that is extremely useful in the mathematical study of
(\ref{EQ:FPE}), and (2) the problem (\ref{1}) itself is a poorly
understood system from numerical aspects when the diffusion tensor
is discontinuous. Therefore, the model problem (\ref{1}) deserves
a study as a research topic in numerical partial differential
equations.

Throughout this paper, we assume that the diffusion tensor
$a(x)=\{a_{ij}(x)\}_{d\times d}\in (L^\infty (\Omega))^{d\times
d}$ is symmetric, uniformly bounded and positive definite in
$\Omega$, and that the drift vector $\bmu \in (L^\infty
(\Omega))^d$. We will follow the usual notation for Sobolev spaces
and norms \cite{ciarlet-fem, Gilbarg-Trudinger, grisvard, gr,
brenner}. For any open bounded domain $D\subset \mathbb{R}^d$ with
Lipschitz continuous boundary, we use $\|\cdot\|_{s,D}$ and
$|\cdot|_{s,D}$ to denote the norm and seminorms in the Sobolev
space $H^s(D)$ for any $s\ge 0$, respectively. The inner product
in $H^s(D)$ is denoted by $(\cdot,\cdot)_{s,D}$. The space
$H^0(D)$ coincides with $L^2(D)$, for which the norm and the inner
product are denoted by $\|\cdot \|_{D}$ and $(\cdot,\cdot)_{D}$,
respectively. When $D=\Omega$, we shall drop the subscript $D$ in
the norm and inner product notation. For convenience, we use
``$\lesssim$ '' to denote ``less than or equal to up to a general
constant independent of the mesh size or functions appearing in
the inequality".

By a {\em weak solution} of (\ref{1}) we mean a function
$u=u(x)\in L^2(\Omega)$ satisfying
\begin{equation}\label{weakform}
(u, {\cal L}v) = - (f, v), \qquad \forall v\in H^2(\Omega)\cap
H_0^1(\Omega),
\end{equation}
where $\CL$ is a differential operator given by ${\cal L}v= \bmu
\cdot\nabla v+\frac{1}{2}\sum_{i,j=1}^da_{ij}\partial_{ji}^2v$.
The differential operator $\CL$ is assumed to satisfy the
$H^2$-regularity property in the sense that for any given $\chi\in
L^2(\Omega)$, there exists a unique strong solution $\Phi\in
H^2(\Omega)\cap H_0^1(\Omega)$ satisfying
\begin{equation}\label{EQ:H2Regularity}
\CL \Phi  = \chi,\qquad \|\Phi\|_2  \lesssim \|\chi\|.
\end{equation}
In \cite{smears}, it was shown that the regularity assumption
(\ref{EQ:H2Regularity}) holds true on bounded convex domain
$\Omega$ if $\bmu=0$ and the diffusion tensor satisfies the
following Cord\`es condition
\begin{equation}\label{cordes}
\frac{\sum_{i,j=1}^d a_{ij}^2}{(\sum_{i=1}^d a_{ii})^2} \leq
\frac{1}{d-1+\varepsilon}\qquad \mbox{in}\ \Omega,
\end{equation}
for a parameter $\varepsilon\in (0,1]$. The Cord\`es condition
(\ref{cordes}) is automatically satisfied in 2D for diffusion
tensor that is bounded, symmetric, and uniformly positive definite
in the domain, see \cite{ww2016} for a verification.

Our numerical scheme for the model problem (\ref{1}) is based on
the weak formulation (\ref{weakform}) through a weak Galerkin
approach that combines the primal variable with its dual. The dual
problem for the primal equation (\ref{weakform}) is given by
\begin{equation}\label{EQ:dual-form}
(w, \CL \rho) = 0,\qquad \forall w\in L^2(\Omega),
\end{equation}
where $\rho\in H^2(\Omega)\cap H_0^1(\Omega)$ is the dual
variable. Under the $H^2$-regularity assumption
(\ref{EQ:H2Regularity}), the solution to the dual problem
(\ref{EQ:dual-form}) is clearly trivial; i.e., $\rho\equiv 0$.
Note that the primal and the dual equations are formally
uncorrelated to each other in the continuous case; but this
changes significantly in the context of weak Galerkin finite
element methods. In the weak Galerkin approach, the differential
operator $\CL$ is discretized as
$$
\CL_w(v):= \bmu \cdot\nabla_w v+\frac{1}{2}\sum_{i,j=1}^d
a_{ij}\partial_{ji,w}^2 v,
$$
where $\nabla_w$ is a discrete weak gradient \cite{wy2013,mwy3655,
wy2707, wy3655} and $\partial_{ji,w}^2$ is a discrete weak second
order partial derivative \cite{mwyz-biharmonic,ww} (also see
Section \ref{Section:Hessian} for their definition). The
corresponding primal and dual equation then become to be
\begin{equation}\label{EQ:discrete-primal}
\mbox{discrete primal:\quad} (u_h, \CL_w v) = - (f, v),\qquad
\forall v \in V_{h,k}^0
\end{equation}
and
\begin{equation}\label{EQ:discrete-dual}
\mbox{discrete dual:\quad}( w, \CL_w \rho_h) = 0 ,\quad\qquad
\forall w \in W_{h,s},
\end{equation}
where $V_{h,k}^0$ and $W_{h,s}$ are two weak finite element spaces
used to approximate $H^2(\Omega)\cap H_0^1(\Omega)$ and
$L^2(\Omega)$ respectively. While neither
(\ref{EQ:discrete-primal}) nor (\ref{EQ:discrete-dual}) makes any
computationally feasible schemes, their combination through the
use of a suitably-defined stabilizer does provide numerical
methods that are efficient, accurate, and stable for several model
problems
\cite{ErikBurman-EllipticCauchy-SIAM01,ErikBurman-EllipticCauchy-SIAM02,ErikBurman-EllipticCauchy,
ww2016}. A formal description of the scheme reads as follows: Find
$u_h\in W_{h,s}$ and $\rho_h\in V_{h,k}^0$ such that
\begin{equation}\label{primal-dual-wg}
\begin{split}
s(\rho_h, v) + (u_h, \CL_w v) &= - (f, v),\qquad
\forall v \in V_{h,k}^0,\\
( w, \CL_w \rho_h) &= 0 ,\quad\qquad \forall
 w \in W_{h,s},
 \end{split}
\end{equation}
where $s(\cdot,\cdot)$ is a bilinear form in $V_{h,k}^0\times
V_{h,k}^0$ known as {\em stabilizer} or {\em smoother} that
enforces a certain weak continuity for the approximation $\rho_h$.
Numerical schemes in the form of (\ref{primal-dual-wg}) were named
{\em primal-dual weak Galerkin finite element methods} in
\cite{ww2016}, but were broadly called {\em stabilized finite
element methods} in
\cite{ErikBurman-EllipticCauchy-SIAM01,ErikBurman-EllipticCauchy-SIAM02,ErikBurman-EllipticCauchy}.

In the rest of the paper, we will provide all the technical
details for the numerical scheme (\ref{primal-dual-wg}), including
the construction of the finite element spaces $V_{h,k}^0$ and
$W_{h,s}$, representation of the stabilizer or smoother
$s(\cdot,\cdot)$, mathematical convergence for the corresponding
numerical approximations, and some numerical results that
demonstrate the performance of the method. One of the
distinguished features of this approach lies on ultra weak
regularity assumptions for the primal variable $u=u(x)$ in the
mathematical convergence theory. The method essentially assumes no
regularity on the primal variable so that solutions with
discontinuity can be well approximated by our primal-dual finite
element method. This work is a non-trivial extension of
\cite{ww2016} in both theory and algorithmic development.

The paper is organized as follows. In Section
\ref{Section:Hessian}, we shall briefly discuss the computation of
weak gradients and weak second order partial derivatives. In
Section \ref{Section:WGFEM}, we will present a detailed
description of the primal-dual weak Galerkin finite element method
for the Fokker-Planck type model problem (\ref{1}) based on the
weak formulation (\ref{weakform}). In Section
\ref{Section:ExistenceUniqueness}, we will study the solution
properties for our numerical method. In particular, we shall
derive an \emph{inf-sup} condition, and then establish a result on
the solution existence and uniqueness. In Section
\ref{Section:error-equations}, we will derive an error equation
for the numerical solutions. Then in Section
\ref{Section:Stability}, we will establish an error estimate for
the primal variable that is of optimal order in $L^2$. Section
\ref{Section:Estimates4L2Projections} is devoted to a presentation
of error estimates for the usual $L^2$ projections. Finally in
Section \ref{Section:numerics}, we report some numerical results
to demonstrate the performance of the primal-dual weak Galerkin
finite element method.

\section{Weak Partial Derivatives}\label{Section:Hessian}
The goal of this section is to brief the definition and
computation of the discrete weak partial derivatives introduced in
\cite{ww, wwhg, wy3655}. To this end, let $T$ be a polygonal or
polyhedral region with boundary $\partial T$. By a weak function
on $T$ we mean a triplet $v=\{v_0,v_b,\textbf{v}_g\}$ in which
$v_0\in L^2(T)$, $v_b\in L^{2}(\partial T)$ and $\textbf{v}_g\in
[L^{2}(\partial T)]^d$. The first and second components $v_0$ and
$v_b$ are intended for the value of $v$ in the interior and on the
boundary of $T$, respectively. The third component
$\textbf{v}_g=(v_{g1},\cdots,v_{gd})\in \mathbb{R}^d$ is used to
represent the gradient of $v$ on $\partial T$. In general, $v_b$
and $\textbf{v}_g$ are not required to be consistent with the
trace of $v_0$ and $\nabla v_0$ on $\partial T$. Denote by $\W(T)$
the space of all weak functions on $T$:
\begin{equation}\label{2.1}
\W(T)=\{v=\{v_0,v_b,\textbf{v}_g\}: v_0\in L^2(T), v_b\in
L^{2}(\partial T), \textbf{v}_g\in [L^{2}(\partial T)]^d\}.
\end{equation}

For any $v\in \W(T)$, the weak second order partial derivative
$\partial^2_{ij}v$, denoted as $\partial^2_{ij,w} v$, is defined
as a linear functional in the dual space of $H^2(T)$ satisfying
 \begin{equation}\label{2.3}
 (\!(\partial^2_{ij,w}v,\varphi)\!)_T=(v_0,\partial^2_{ji}\varphi)_T-
 \langle v_b n_i,\partial_j\varphi\rangle_{\partial T}+
 \langle v_{gi},\varphi n_j\rangle_{\partial T},
 \end{equation}
for all $\varphi\in H^2(T)$. Here, $(\!( \chi, \varphi)\!)_T$
stands for the action of $\chi$ at $\varphi\in H^2(T)$,
$\textbf{n}=(n_1,\cdots,n_d)$ is the unit outward normal direction
to $\partial T$, $(\cdot,\cdot)_T$ is the usual $L^2$ inner
product in $L^2(T)$, and $\langle \cdot, \cdot\rangle_{\pT}$ is
the $L^2$ inner product in $L^2(\pT)$.

The weak gradient of $v\in \W(T)$, denoted by $\nabla_w v$, is
defined as a linear functional in the dual space of $[H^1(T)]^d$
such that
\begin{equation*}
\action{\nabla_w v}{\boldsymbol{\psi}}
=-(v_0,\nabla \cdot
\boldsymbol{\psi})_T+\langle v_b,\boldsymbol{\psi}\cdot
\textbf{n}\rangle_{\partial T},
\end{equation*}
for all $\boldsymbol{\psi}\in [H^1(T)]^d$. Note that the weak
gradient makes no use of the third component of the weak function
$v$.

Denote by $P_r(T)$ the set of polynomials on $T$ with degree no
more than $r$. A discrete version of $\partial^2_{ij,w} v$ for
$v\in \W(T)$, denoted by $\partial^2_{ij,w,r,T} v$, is defined as
the unique polynomial
 in $P_r(T)$ satisfying
  \begin{equation}\label{2.4}
 (\partial^2_{ij,w,r,T}v,\varphi)_T=(v_0,\partial^2
 _{ji}\varphi)_T-\langle v_b n_i,\partial_j\varphi\rangle_{\partial T}
 +\langle v_{gi},\varphi n_j\rangle_{\partial T},
 \end{equation}
for all $ \varphi \in
 P_r(T)$, which, by using integration by parts, yields
  \begin{equation}\label{2.4*}
 (\partial^2_{ij,w,r,T}v,\varphi)_T=(\partial^2
 _{ ij}v_0,\varphi)_T+\langle (v_0-v_b) n_i,\partial_j\varphi\rangle_{\partial T}
 -\langle \partial_i v_0-v_{gi},\varphi n_j\rangle_{\partial T},
 \end{equation}
for all $ \varphi \in P_r(T)$.

A discrete form of $\nabla_{w} v$ for $v\in \W(T)$, denoted
by $\nabla_{w,r,T}v$, is defined as the unique polynomial vector
in $[P_r(T) ]^d$ satisfying
\begin{equation}\label{disgradient}
(\nabla_{w,r,T}  v, \boldsymbol{\psi})_T=-(v_0,\nabla \cdot
\boldsymbol{\psi})_T+\langle v_b, \boldsymbol{\psi} \cdot
\textbf{n}\rangle_{\partial T}, \quad\forall\boldsymbol{\psi}\in
[P_r(T)]^d,
\end{equation}
 which, from integration by parts, gives
\begin{equation}\label{disgradient*}
(\nabla_{w,r,T}  v, \boldsymbol{\psi})_T= (\nabla v_0,
\boldsymbol{\psi})_T-\langle v_0- v_b, \boldsymbol{\psi} \cdot
\textbf{n}\rangle_{\partial T}, \quad\forall\boldsymbol{\psi}\in
[P_r(T)]^d,
\end{equation}
provided that $v_0$ is sufficiently regular.

\section{Numerical Schemes}\label{Section:WGFEM}
The goal of this section is to present a finite element method for
the variational problem (\ref{weakform}). To this end, let ${\cal
T}_h$ be a finite element partition of the domain $\Omega$ into
polygons in 2D or polyhedra in 3D which is shape regular in the
sense as described in \cite{wy3655}. For three dimensional
domains, all the polyhedral elements are assumed to have flat
faces. Denote by ${\mathcal E}_h$ the set of all edges or faces in
${\cal T}_h$ and  ${\mathcal E}_h^0={\mathcal E}_h \setminus
\partial\Omega$ the set of all interior edges or faces.
Denote by $h_T$ the meshsize of $T\in {\cal T}_h$ and
$h=\max_{T\in {\cal T}_h}h_T$ the meshsize of the partition
${\cal T}_h$.

Let $k\geq 1$ be a given integer. Denote by
$V_k(T)$ the discrete local weak function space on $T$ given by
$$
V_k(T)=\{\{v_0,v_b,\textbf{v}_g\}:\ \ v_0\in P_k(T),\ v_b\in
P_k(e),\ \textbf{v}_g\in [P_{k-1}(e)]^d,\ e\subset \partial T\}.
$$
By patching $V_k(T)$ over all the elements $T\in {\cal T}_h$
through a common value $v_b$ and $\textbf{v}_g$ on the interior
edges/faces, we obtain a global weak finite element space
$V_{h,k}$:
$$
V_{h,k}=\big\{\{v_0,v_b,\textbf{v}_g\}:\ \ \{v_0,v_b,\textbf{v}_g\}|_T\in
V_k(T), \ T\in {\cal T}_h
 \big\}.
$$
Denote by $V_{h,k}^0$ the subspace of $V_{h,k}$ with vanishing
boundary value for $v_b$ on $\partial \Omega$,  i.e.,
$$
V_{h,k}^0=\big\{\{v_0,v_b,\textbf{v}_g\}\in V_{h,k}:\ \   v_b|_e=0, \ e\subset \partial\Omega
 \big\}.
$$

For any given integer $s\ge 0$, denote by $W_{h,s}$ the usual
finite element space consisting of piecewise polynomials of degree
$s$; i.e.,
$$
W_{h,s}=\{w:\ \  w|_T\in P_{s}(T),\ T\in {\cal T}_h\}.
$$
For application in the approximation of (\ref{weakform}), the
integer $s$ will be chosen as either $s=k-1$ or $s=k-2$. In the
case of $k=1$, the only viable option for this integer would be
$s=0$.

For simplicity of notation, denote by $\nabla_{w}\sigma$ the discrete weak gradient
$\nabla_{w,k-1,T}\sigma$ computed by using (\ref{disgradient}) on each element $T$ with $r=k-1$:
$$
(\nabla_{w}\sigma)|_T= \nabla_{w,k-1,T}(\sigma|_T), \qquad \sigma\in V_{h,k}.
$$
Analogously, we use $\partial^2_{ij, w}\sigma$ to denote the discrete weak second order partial derivative $\partial^2_{ij,w,s,T}\sigma$ computed by using (\ref{2.4}) on each element $T$ with $r=s$:
$$
(\partial^2_{ij, w} \sigma)|_T=\partial^2_{ij,w,s,T}(\sigma|_T),\ \ \ \sigma\in V_{h,k}.
$$
The corresponding weak differential operator is defined by using weak partial derivatives as follows
 $$
{\cal L}_w(\sigma) = \mu\cdot \nabla_w\sigma+\frac{1}{2}\sum_{i,j=1}^da_{ij}\partial_{ji,w}^2\sigma,
$$
for any $\sigma\in V_{h,k}$.

Let us introduce the following bilinear forms
\begin{align*}
s(\rho, \sigma)=&\sum_{T\in {\cal T}_h}s_T(  \rho, \sigma),\qquad \rho, \sigma\in V_{h,k},\\
b(\sigma, v)=&\sum_{T\in {\cal T}_h}b_T(\sigma, v), \qquad \sigma\in V_{h,k},\ v\in W_{h,s},
\end{align*}
where
\begin{equation}\label{EQ:local-stabilizer}
 \begin{split}
 s_T( \rho, \sigma)=&h_T^{-3}\int_{\partial T}
(\rho_0-\rho_b)(\sigma_0-\sigma_b)ds \\
& + h_T^{-1}\int_{\partial T}(\nabla \rho_0 -\brho_g ) (\nabla \sigma_0 -\bsigma_g )ds\\
&+\delta\int_T{\cal L}(\rho_0){\cal L}(\sigma_0) dT
\end{split}
\end{equation}
and
$$
b_T(\sigma, v)=(v, {\cal L}_w(\sigma) )_T.
$$
Here, $\delta >0 $ is a parameter independent of the meshsize $h$ and the functions involved.

We are now in a position to state our primal-dual weak Galerkin
finite element scheme for the model variational problem
(\ref{weakform}).
 \begin{algorithm} Let $k\ge 1$ be a given integer and $s\ge 0$ be another integer.
A numerical approximation for the solution of (\ref{weakform}) is
the component $u_h$ in $(u_h;\rho_h)\in W_{h,s} \times V^0_{h,k}$
satisfying
 \begin{eqnarray}\label{32}
s( \rho_h, \sigma)+b(\sigma, u_h)&=&-(f, \sigma_0),\qquad \forall\sigma\in V^0_{h,k},\\
b(\rho_h, v)&=&0,\quad\qquad\qquad \forall v\in W_{h,s}.\label{2}
\end{eqnarray}
\end{algorithm}

\section{Solution Existence, Uniqueness, and Stability}\label{Section:ExistenceUniqueness}
The goal of this section is to study the solution for the numerical scheme (\ref{32})-(\ref{2}). In particular, we shall prove the existence and uniqueness for the numerical solution under certain assumptions on the finite element partition $\T_h$ and the differential operator ${\cal L}$.

On each element $T\in\T_h$, denote by $Q_0$ the $L^2$ projection onto
$P_k(T)$, $k\geq 1$. Similarly, on each edge or face $e\subset\partial T$,
denote by $Q_b$ and $\textbf{Q}_g:=(Q_{g1},\cdots, Q_{gd})$ the $L^2$ projections onto $P_{k}(e)$ and
$[P_{k-1}(e)]^d$, respectively. For any $w\in H^2(\Omega)$, define the projection $Q_h w\in V_{h,k}$ so that on each element $T$ one has
\begin{equation}\label{EQ:OperatorQh}
Q_hw=\{Q_0w,Q_bw,\textbf{Q}_g(\nabla w)\}.
\end{equation}
Denote by $\CQ_h^{(s)}$ the $L^2$ projection onto $W_{h,s}$ - the space of piecewise polynomials of degree $s\ge 0$. In the rest of this paper, the integer $s$ will be taken as either $s=k-1$ or $s=k-2$.

\begin{lemma}\label{Lemma5.1} \cite{ww,wwhg,wy3655} The aforementioned projection operators satisfy the following commutative properties: For any $w\in H^2(T)$, one has
\begin{align}\label{l}
\partial^2_{ij,w}(Q_h w) = &{\cal Q}^{(s)}_h(\partial^2_{ij} w),\qquad i,j=1,\ldots,d,\\
 \nabla_{ w}(Q_h w) = &{\cal Q}^{(k-1)}_h( \nabla w).\label{l-2}
\end{align}
\end{lemma}
\begin{proof} For any $\varphi\in  P_{s}(T)$ and  $w\in H^2(T)$, from (\ref{2.4}), the usual property of $L^2$ projections, and the integration by parts, we have
\begin{equation*}
\begin{split}
(\partial^2_{ij,w}(Q_h w),\varphi)_T&=(Q_0
w,\partial^2_{ji}\varphi)_T -\langle Q_b w,\partial_j \varphi\cdot
n_i\rangle_{\partial T}+
\langle Q_{gi}(\partial_i w)\cdot n_j,\varphi\rangle_{\partial T}\\
&=(w,\partial^2_{ji}\varphi)_T-\langle w,\partial_j
 \varphi\cdot n_i\rangle_{\partial T}+
 \langle \partial_i w\cdot n_j,\varphi\rangle_{\partial T}\\
&=(\partial^2_{ij}w,\varphi)_T\\
&=({\cal Q}^{(s)}_h\partial^2_{ij}w,\varphi)_T,
\end{split}
\end{equation*}
which completes the proof of (\ref{l}). The other identity (\ref{l-2}) can be derived in a similar fashion, and details can be found in \cite{ww,wwhg,wy3655}.
\end{proof}

The stabilizer $s(\cdot,\cdot)$ defined through (\ref{EQ:local-stabilizer}) naturally induces a semi-norm in the weak finite element space $V_{h,k}$ as follows:
\begin{equation}\label{norm-new}
\3bar \rho  \3bar_w=
s(\rho, \rho)^{\frac{1}{2}},\qquad  \rho\in V_{h,k}.
\end{equation}

\begin{lemma}\label{lem3-new} ({\it inf-sup} condition) Assume that the drift term $\mu\in L^\infty(\Omega)$ and the coefficient tensor $a(x)$ is uniformly piecewise continuous with respect to the finite element partition $\T_h$. Then, there exists a constant $\beta>0$ such that
for any $v\in W_{h,s}$, there exists a weak function $\sigma\in V_{h,k}^0$ satisfying
\begin{eqnarray}\label{EQ:inf-sup-condition-01}
b(v,\sigma) & \geq & \frac12 \|v\|^2,\\
\3bar \sigma \3bar_w & \leq & \beta \|v\|, \label{EQ:inf-sup-condition-02}
\end{eqnarray}
provided that meshsize satisfies $h\le h_0$ for a small, but fixed parameter value $h_0>0$.
\end{lemma}

\begin{proof}
Let $\Phi$ be the solution of the following auxiliary problem:
\begin{eqnarray}\label{pro-new}
{\cal L}\Phi & = & v, \qquad \text{in}\ \Omega,\\
\Phi & = & 0, \qquad\text{on}\ \partial\Omega.\label{pro0-new}
\end{eqnarray}
From the assumption (\ref{EQ:H2Regularity}), the problem
(\ref{pro-new})-(\ref{pro0-new}) has the following
$H^2$-regularity estimate
\begin{equation}\label{regu2}
\|\Phi\|_{2}\leq C\|v\|.
\end{equation}
With $\sigma=Q_h\Phi$, we have from Lemma \ref{Lemma5.1} that
\begin{equation}\label{EQ:April05:100}
\begin{split}
b(v,\sigma) = & \sum_{T\in {\cal T}_h}(v, {\cal L}_w (Q_h\Phi))_T \\
=& \sum_{T\in {\cal T}_h}(\bmu v, \nabla_w Q_h\Phi)_T+
\frac{1}{2} \sum_{i,j=1}^d(a_{ij} v,\partial_{ji,w}^2Q_h\Phi)_T\\
=&\sum_{T\in {\cal T}_h}(\bmu v, {\cal Q}^{(k-1)}_h(\nabla \Phi))_T+
\frac{1}{2} \sum_{i,j=1}^d(a_{ij} v,{\cal Q}^{(s)}_h\partial_{ji}^2\Phi)_T\\
=&\sum_{T\in {\cal T}_h}(\bmu v, \nabla \Phi)_T+ \frac{1}{2}
\sum_{i,j=1}^d(a_{ij}v, \partial_{ji}^2\Phi)_T\\
&+(\bmu v, ( {\cal Q}^{(k-1)}_h-I) \nabla \Phi )_T+
\frac{1}{2} \sum_{i,j=1}^d(a_{ij}v,({\cal Q}^{(s)}_h-I)\partial_{ji}^2\Phi)_T\\
=&\sum_{T\in {\cal T}_h}(v,v)_T+\sum_{T\in {\cal T}_h}
 (\bmu v, ( {\cal Q}^{(k-1)} _h-I) \nabla \Phi)_T\\
 &+\frac{1}{2}\sum_{i,j=1}^d ((a_{ij}-\bar{a}_{ij})v,
 ({\cal Q}^{(s)}_h-I)\partial_{ij}^2 \Phi)_T,
\end{split}
\end{equation}
where $\bar{a}_{ij}$ stands for the average of
$a_{ij}$ on each element $T\in {\cal T}_h$. As the coefficient tensor $a(x)=(a_{ij}(x))_{d\times d}$ is uniformly piecewise continuous in $\Omega$, there exists a small parameter $\varepsilon(h)$ depending on the meshsize $h$ and the continuity of $a_{ij}$ on each element $T$ such that
\begin{equation*}
\begin{split}
&\left|\sum_{T\in {\cal T}_h} (\bmu v, ( {\cal Q}^{(k-1)} _h-I) \nabla \Phi)_T\right|
\leq C h \|\Phi\|_{2}\|v\|\\
&\left|\sum_{i,j=1}^d ((a_{ij}-\bar{a}_{ij})v,({\cal Q}^{(s)}_h-I)\partial_{ij}^2 \Phi)_T\right| \leq C \varepsilon(h) \|\Phi\|_{2}\|v\|.
\end{split}
\end{equation*}
Substituting the above estimates into (\ref{EQ:April05:100}) yields
\begin{equation}\label{EQ:April05:101}
\begin{split}
b(v,\sigma) & \geq \|v\|^2-C(h+\varepsilon(h)) \|\Phi\|_{2}\|v\|\\
& \geq (1-C h - C\varepsilon(h))\|v\|^2,
\end{split}
\end{equation}
where we have used the regularity estimate (\ref{regu2}). In particular, for $\varepsilon_0=\frac{1}{2C}$, there exists a parameter value $h_0$ such that $ h+\varepsilon(h)\leq \varepsilon_0$ when $h\le h_0$. It follows that $(h+\varepsilon(h)) C \le \frac12$ holds true for $h\le h_0$, and hence
\begin{equation}\label{EQ:April05:102}
b(v,\sigma) \geq \frac12 \|v\|^2,
\end{equation}
which verifies the inequality (\ref{EQ:inf-sup-condition-01}).

It remains to establish the estimate (\ref{EQ:inf-sup-condition-02}) for $\sigma=Q_h\Phi$. To this end, from the usual trace inequality we have
\begin{equation}\label{EQ:Estimate:002}
\begin{split}
&\sum_{T\in {\cal T}_h }h_T^{-3}\int_{\partial
T}|\sigma_0-\sigma_b|^2ds\\
=&\sum_{T\in {\cal T}_h}
h_T^{-3}\int_{\partial T}|Q_0\Phi-Q_b \Phi|^2ds\\
\leq &\sum_{T\in {\cal T}_h} h_T^{-3}\int_{\partial T}|Q_0\Phi-\Phi|^2ds\\
\lesssim & \sum_{T\in {\cal T}_h} h_T^{-4}\int_{T}|Q_0\Phi-\Phi|^2dT
+h_T^{-2}\int_{T}|\nabla Q_0\Phi-\nabla \Phi|^2dT\\
\lesssim & \|\Phi\|_{2}^2 \lesssim \|v\|^2.
 \end{split}
\end{equation}
A similar analysis can be applied to yield the following estimate:
\begin{equation}\label{EQ:Estimate:003}
\begin{split}
\sum_{T\in {\cal T}_h }h_T^{-1}\int_{\partial T}|\nabla
\sigma_0-\bsigma_g|^2ds \lesssim & \|v\|^2.
 \end{split}
\end{equation}
Furthermore, we have
\begin{equation}\label{EQ:Estimate:004}
\begin{split}
\sum_{T\in {\cal T}_h } \delta \int_{T}| {\cal L}(\sigma_0)|^2dT
\lesssim &  \sum_{T\in {\cal T}_h } \| \sigma_0\|_{2,T}^2\\
\lesssim & \sum_{T\in {\cal T}_h } \| Q_0 \Phi\|_{2,T}^2\\
\lesssim &  \|\Phi\|_{2 }^2
\lesssim
\|v\|^2.
 \end{split}
\end{equation}
Finally, by combining the estimates (\ref{EQ:Estimate:002})-(\ref{EQ:Estimate:004}) and the definition of $\3bar \sigma \3bar_w$ we obtain
\begin{equation*}\label{barv}
  \3bar \sigma\3bar_w^2 \leq \beta^2 \|v\|^2
\end{equation*}
for some constant $\beta$. This completes the proof of the lemma.
\end{proof}

We are now in a position to state the main result on solution existence and uniqueness.

\begin{theorem}\label{thmunique1} Assume that the drift vector $\bmu$ and the coefficient tensor $a(x)$ are uniformly piecewise continuous in $\Omega$ with respect to the finite element partition ${\cal T}_h$. Then, there exists a fixed $h_0>0$ such that the primal-dual weak Galerkin finite element algorithm (\ref{32})-(\ref{2}) has one and only one solution if the meshsize satisfies $h\le h_0$.
\end{theorem}

\begin{proof} It suffices to show that zero is the only solution to the problem (\ref{32})-(\ref{2}) with homogeneous data $f=0$ and $g=0$. To this end, assume $f=0$ and $g=0$ in (\ref{32})-(\ref{2}). By choosing $v=u_h$ and $\sigma=\rho_h$, the difference of (\ref{2}) and (\ref{32}) gives
$s(\rho_h,\rho_h)=0$, which implies $\rho_0=\rho_b$ and $\nabla \rho_0=\brho_g$ on each
$\partial T$, and hence $\rho_0\in C^1_0(\Omega)$. Moreover, we have
$$
{\cal L}\rho_0=0.
$$
Thus, from the solution existence and uniqueness for the differential operator ${\cal L}$ with homogeneous Dirichlet boundary condition we obtain $\rho_0\equiv 0$, and hence $\rho_h\equiv 0$.

With $\rho_h=0$, the equation (\ref{32}) now becomes
\begin{equation}\label{rew}
b(u_h, \tau) = 0,\qquad \forall \tau\in V_{h,k}^0.
\end{equation}
From Lemma \ref{lem3-new}, there exists a $\sigma\in V_{h,k}^0$ such that
$b(u_h, \sigma) \ge \frac12 \|u_h\|^2$. It then follows from (\ref{rew}) that $u_h\equiv 0$.
This completes the proof of the theorem.
\end{proof}

\section{Error Equations}\label{Section:error-equations}
In this section we shall derive an error equation for the numerical solution arising from the primal-dual weak Galerkin finite element algorithm (\ref{32})-(\ref{2}). To this end, let $u$ and $(u_h;\rho_h) \in W_{h,s}\times V_{h,k}^0$ be the solution of (\ref{weakform}) and (\ref{32})-(\ref{2}), respectively. Note that $\rho_h$ is supposed to approximate the trivial function $\rho=0$.

\begin{lemma}\label{Lemma:LocalEQ} For any $\sigma\in V_{h,k}$ and $v\in W_{h,s}$, the following identity holds true:
\begin{equation}
({\cal L}_w \sigma, v)_T = ({\cal L} \sigma_0, v)_T +R_T(\sigma, v),
\end{equation}
where
\begin{equation}\label{EQ:April:06:100}
\begin{split}
R_T(\sigma,v) = & \frac{1}{2}\sum_{i,j=1}^d\langle \sigma_0-\sigma_b, n_j
\partial_i({\cal Q}_h^{(s)}(a_{ij} v))\rangle_{\partial T}\\
&-\frac{1}{2}\sum_{i,j=1}^d\langle
\partial_j\sigma_0-\sigma_{gj}, n_i {\cal Q}^{(s)}_h(a_{ij} v)\rangle_{\partial T}\\
& - \langle \sigma_0-\sigma_b, {\cal Q}^{(k-1)}_h(\bmu
v)\cdot \bn\rangle_{\partial T}.
\end{split}
\end{equation}
\end{lemma}

\begin{proof}
From the formula (\ref{2.4*}) and (\ref{disgradient*})
for the weak derivatives, we have
\begin{equation}\label{EQ:April:06:200}
\begin{split}
 &({\cal L}_w (\sigma), v)_T \\
=&(\bmu\cdot\nabla_w \sigma, v)_T+
\frac{1}{2}\sum_{i,j=1}^d (a_{ij} \partial_{ji,w}^2\sigma, v)_T\\
=&(\nabla_w\sigma, {\cal Q}^{(k-1)}_h(\bmu v))_T+
\frac{1}{2}\sum_{i,j=1}^d (\partial_{ji,w}^2\sigma, {\cal Q}_h^{(s)}(a_{ij}v))_T\\
=& (\nabla \sigma_0, \bmu v)_T +\frac{1}{2}\sum_{i,j=1}^d
(\partial_{ji}^2 \sigma_0, a_{ij}v)_T +R_T(\sigma,v)\\
= & ({\cal L} \sigma_0, v)_T +R_T(\sigma,v),
\end{split}
\end{equation}
where $R_T(\sigma,v)$ is given by (\ref{EQ:April:06:100}).
\end{proof}

By \emph{error functions} we mean the difference between the numerical solution and the $L^2$ interpolation of the exact solution; namely,
\begin{align}\label{error}
e_h&=u_h-{\cal Q}^{(s)}_hu,\\
\epsilon_h&=\rho_h-Q_h\rho=\rho_h.\label{error-2}
\end{align}

\begin{lemma}\label{errorequa}
Let $u$ and $(u_h;\rho_h) \in W_{h,s}\times V_{h,k}^0$ be the solutions arising from
(\ref{1}) and (\ref{32})-(\ref{2}), respectively.  Then, the error functions $e_h$
and $\epsilon_h$ satisfy the following equations
\begin{eqnarray}\label{sehv}
 s( \epsilon_h , \sigma)+b(\sigma, e_h )&=&\ell_u(\sigma)
,\qquad \forall\sigma\in V_{h,k}^0,\\
b(\epsilon_h, v)&=&0,\qquad\qquad \forall v\in W_{h,s}, \label{sehv2}
\end{eqnarray}
where $\ell_u(\sigma)$ is given by
\begin{equation}\label{lu}
\begin{split}
\qquad \ell_u(\sigma) =&\sum_{T\in {\cal T}_h}\langle
\sigma_b-\sigma_0 ,  (\bmu  u-{\cal Q}^{(k-1)}_h(\bmu {\cal Q}^{(s)}_hu))\cdot
\bn\rangle_{\partial
T}\\
&+\frac{1}{2}\sum_{T\in {\cal T}_h}\sum_{i,j=1}^d\langle
\sigma_0 -\sigma_b, n_j
\partial_i (a_{ij}u-{\cal Q}^{(s)}_h(a_{ij}{\cal Q}^{(s)}_hu))\rangle_{\partial T}\\
&-\frac{1}{2}\sum_{T\in {\cal T}_h}\sum_{i,j=1}^d\langle (
\partial_j\sigma_0 -\sigma_{gj})n_i , a_{ij}  u-\CQ_h^{(s)}(a_{ij} \CQ_h^{(s)}u)\rangle_{\partial T}\\
&+\sum_{T\in {\cal T}_h}({\CL}(\sigma_0), {\CQ}_h^{(s)}u-u)_T.
\end{split}
\end{equation}
\end{lemma}

\begin{proof} First of all, from (\ref{error-2}) and (\ref{2}) we have
\begin{align*}
 b(\epsilon_h, v) = b(\rho_h, v) = 0,\qquad\qquad \forall v\in W_{h,s},
 \end{align*}
which gives (\ref{sehv2}).

Next, notice that $\rho=0$. Thus, by using (\ref{32}) we arrive at
\begin{equation}\label{EQ:April:04:100}
\begin{split}
 & s( \rho_h-Q_h\rho, \sigma)+b(\sigma, u_h-{\cal Q}^{(s)}_hu) \\
 = & s( \rho_h, \sigma)+b(\sigma, u_h)- b(\sigma, {\cal Q}^{(s)}_hu) \\
  = &-(f, \sigma_0)-b(\sigma,{\cal Q}^{(s)}_hu).
\end{split}
\end{equation}
The rest of the proof shall deal with the term
$b(\sigma,{\cal Q}^{(s)}_hu)$. To this end, we use Lemma \ref{Lemma:LocalEQ} to obtain
\begin{equation}\label{EQ:April:04:103}
\begin{split}
 b(\sigma,{\cal Q}^{(s)}_hu) = & \sum_{T\in {\cal T}_h} ({\cal L}_w \sigma, {\cal Q}_h^{(s)}u )_T \\
=& \sum_{T\in {\cal T}_h} ({\cal L} \sigma_0, {\cal Q}_h^{(s)}u )_T +R_T(\sigma, {\cal Q}_h^{(s)}u) \\
= & \sum_{T\in {\cal T}_h} ({\cal L} \sigma_0, u)_T + ({\cal L} \sigma_0, {\cal Q}_h^{(s)}u - u)_T +R_T(\sigma, {\cal Q}_h^{(s)}u).
\end{split}
\end{equation}
From the integration by parts we have
\begin{equation}\label{EQ:April:04:104}
\begin{split}
 \sum_{T\in {\cal T}_h} ({\cal L} \sigma_0, u)_T = &
 \sum_{T\in {\cal T}_h}\left((\bmu\cdot \nabla \sigma_0,  u )_T +\frac{1}{2}\sum_{i,j=1}^d
(a_{ij}\partial_{ji}^2 \sigma_0, u)_T\right) \\
=&\sum_{T\in {\cal T}_h}(\sigma_0, -\nabla\cdot (\bmu
u)+\frac{1}{2}\sum_{i,j=1}^d\partial_{ij}^2(a_{ij} u ))_T \\
& +\sum_{T\in {\cal T}_h}\langle \sigma_0 , \bmu u\cdot \bn -\frac{1}{2}\sum_{i,j=1}^d n_j\partial_i (a_{ij} u) \rangle_{\partial T}\\
& +\sum_{T\in {\cal T}_h}  \sum_{j=1}^d\langle \partial_j\sigma_0, \frac{1}{2}\sum_{i=1}^d a_{ij} n_i u\rangle_{\partial T}.
\end{split}
\end{equation}
As $u$ is the exact solution of (\ref{weakform}) and $\sigma_b=0$ on $\partial\Omega$, we then have
\begin{eqnarray}\label{EQ:April:04:105}
 \sum_{T\in {\cal T}_h}\langle \sigma_b , \bmu u\cdot \bn -\frac{1}{2}\sum_{i,j=1}^d n_j\partial_i (a_{ij} u) \rangle_{\partial T} & = & 0,\\
\sum_{T\in {\cal T}_h}  \sum_{j=1}^d\langle \sigma_{gj},
\frac{1}{2}\sum_{i=1}^d a_{ij} n_i u\rangle_{\partial T} & = &
0.\label{EQ:April:04:106}
\end{eqnarray}
By combining (\ref{EQ:April:04:104}) with (\ref{EQ:April:04:105}) and (\ref{EQ:April:04:106}) we arrive at
\begin{equation}\label{EQ:April:04:107}
\begin{split}
 \sum_{T\in {\cal T}_h} ({\cal L} \sigma_0, u)_T =
& -(\sigma_0, f) +\sum_{T\in {\cal T}_h}\langle \sigma_0-\sigma_b , \bmu u\cdot \bn -\frac{1}{2}\sum_{i,j=1}^d n_j\partial_i (a_{ij} u) \rangle_{\partial T}\\
& +\sum_{T\in {\cal T}_h}  \sum_{j=1}^d\langle \partial_j\sigma_0
- \sigma_{gj}, \frac{1}{2}\sum_{i=1}^d a_{ij} n_i
u\rangle_{\partial T}.
\end{split}
\end{equation}
Finally, substituting (\ref{EQ:April:04:107}) and (\ref{EQ:April:04:103}) into (\ref{EQ:April:04:100}) gives rise to the error equation (\ref{sehv}) after regrouping the reminding terms. This completes the proof of the lemma.
\end{proof}

\section{Error Estimates}\label{Section:Stability}
The goal of this section is to establish some error estimates for the numerical solutions arising from the primal-dual weak Galerkin finite element scheme (\ref{32})-(\ref{2}). The key to our error analysis is the error equations (\ref{sehv})-(\ref{sehv2}) and the \emph{inf-sup} condition derived in Lemma \ref{lem3-new}.

\subsection{Main results} We first show that the semi-norm induced from the stabilizer $s(\cdot,\cdot)$ is indeed a norm in the subspace $V_{h,k}^0$ consisting of weak functions with vanishing boundary value.

\begin{lemma} The semi-norm $\3bar \cdot \3bar_w$ given in (\ref{norm-new}) defines a norm in the linear space $V_{h,k}^0$.
\end{lemma}

\begin{proof}
It suffices to verify the positivity property for $\3bar\cdot\3bar_w$. Let $\rho\in V_{h,k}^0$ be such that $\3bar \rho \3bar_w=0$. It follows that $s(\rho,\rho)=0$, and hence
$\rho_0=\rho_b$ and $\nabla \rho_0 =\brho_g$ on each
$\partial T$. Furthermore, we have ${\cal L}(\rho_0)=0$ on each element
$T\in {\cal T}_h$.  Thus, $\rho_0\in C^1_0(\Omega)$ and satisfies
$$
 {\cal L} \rho_0  =0,\qquad \mbox{in} \ \Omega.
$$
It follows that $\rho_0 \equiv 0$, and hence $\rho\equiv 0$. This completes the proof.
\end{proof}

 Our main error estimate can be stated as follows.
\begin{theorem} \label{theoestimate}
Let $u$ and $(u_h;\rho_h) \in W_{h,s}\times V_{h,k}^0$ be the solutions of (\ref{weakform}) and (\ref{32})-(\ref{2}), respectively. Let $m\in [2, k+1]$ if $s=k-1$ and $m\in [2,k]$ if $s=k-2$.
Assume that the coefficient tensor $a(x)$ and the drift vector $\bmu$ are uniformly piecewise smooth up to order $m-1$ in $\Omega$ with respect to the finite element partition ${\cal T}_h$.
Additionally, assume that the exact solution $u$ is sufficiently regular such that $u\in H^{m-1}(\Omega)\cap H^2(\Omega)$. Then, we have
 \begin{equation}\label{erres}
\3bar \rho_h \3bar_w+\|u_h - {\cal Q}^{(s)}_h u\| \lesssim h^{m-1}(\|u\|_{m-1}+h\delta_{m,2}\|u\|_2),
\end{equation}
provided that the meshsize $h<h_0$ holds true for a sufficiently small, but fixed $h_0>0$. Moreover, one has the following optimal order error estimate in $L^2$:
\begin{equation}\label{erres:L2}
\|u - u_h\| \lesssim h^{m-1}(\|u\|_{m-1}+h\delta_{m,2}\|u\|_2).
\end{equation}
\end{theorem}

\begin{proof} By letting $\sigma = \epsilon_h$ in the error equation (\ref{sehv}) and using
(\ref{sehv2}) we arrive at
\begin{equation}\label{EQ:April7:001}
s(\epsilon_h, \epsilon_h) = \ell_u(\epsilon_h).
\end{equation}
To deal with the term on the right-hand side of (\ref{EQ:April7:001}), we use the Cauchy-Schwarz inequality, the representation (\ref{lu}), and the estimates (\ref{error3})-(\ref{term3})
to obtain
\begin{equation*}\label{aij}
\begin{split}
&\ |\ell_u(\sigma)| \\
\leq & \Big(\sum_{T\in {\cal
T}_h}h_T^{-3}\|\sigma_0-\sigma_b\|^2_{\partial
T}\Big)^{\frac{1}{2}} \Big(\sum_{T\in {\cal T}_h}h_T^{ 3}\| (\bmu
u-{\cal Q}^ {(k-1)}_h(\bmu {\cal Q}^ {(s)}_hu))\cdot
\bn\|^2_{\partial T}\Big)^{\frac{1}{2}}\\
 + & \Big(\sum_{T\in {\cal T}_h}\sum_{j=1}^dh_T^{-3}\|
\sigma_0-\sigma_b\|^2_{\partial T}\Big)^{\frac{1}{2}}
\Big(\sum_{T\in {\cal T}_h}\sum_{i,j=1}^dh_T^{ 3}\|\partial_i
(a_{ij}u-{\cal Q}_h^{(s)}(a_{ij}{\cal Q}^ {(s)}_hu))
\|^2_{\partial T}\Big)^{\frac{1}{2}}\\
 +&  \Big(\sum_{T\in {\cal
T}_h}\sum_{j=1}^dh_T^{-1}\|\partial_j\sigma_0-\sigma_{gj}\|^2_{\partial
T}\Big)^{\frac{1}{2}} \Big(\sum_{T\in {\cal T}_h}\sum_{i,j=1}^dh_T
\| a_{ij}u-{\cal Q}_h^{(s)}(a_{ij}{\cal Q}^ {(s)}_hu)\|^2_{\partial
T}\Big)^{\frac{1}{2}}\\
 +& \Big(\sum_{T\in {\cal
T}_h}\|{\cal L}(\sigma_0) \|^2_{T}\Big)^{\frac{1}{2}} \Big(\sum_{T\in {\cal T}_h}
\| {\cal Q}_h^{(s)} u-u\|^2_{
T}\Big)^{\frac{1}{2}}\\
\lesssim & h^{k}\|u\|_{k-1} \3bar\sigma \3bar_w  + h^{k-1}(\|u\|_{k-1}+h \delta_{k,2}\|u\|_2) \3bar\sigma \3bar_w \\
&+ h^{k-1}\|u\|_{k-1} \3bar\sigma \3bar_w  + h^{k-1}\|u\|_{k-1} \3bar\sigma \3bar_w  \\
\lesssim &  h^{k-1}(\|u\|_{k-1}+h\delta_{k,2}\|u\|_2) \3bar\sigma \3bar_w,
\end{split}
\end{equation*}
for any $\sigma\in V_{h,k}^0$. Now substituting the above estimate into
(\ref{EQ:April7:001}) yields the following error estimate:
\begin{equation*}
\3bar \epsilon_h \3bar_w^2 \lesssim h^{k-1}(\|u\|_{k-1}+h\delta_{k,2}\|u\|_2) \3bar\epsilon_h \3bar_w,
\end{equation*}
which leads to
\begin{equation}\label{EQ:April7:002}
\3bar \epsilon_h \3bar_w \lesssim h^{k-1}(\|u\|_{k-1}+h\delta_{k,2}\|u\|_2).
\end{equation}

Next, for the error function $e_h=u_h - {\cal Q}^{(k-2)}_h u$, from the \emph{inf-sup}
condition (\ref{EQ:inf-sup-condition-01})-(\ref{EQ:inf-sup-condition-02}) there exists a $\sigma\in V_{h,k}^0$ such that
\begin{equation}\label{EQ:April7:005}
\frac12 \|e_h\|^2 \leq | b(e_h, \sigma) |,\quad \3bar \sigma\3bar_w \leq \beta \|e_h\|.
\end{equation}
On the other hand, the error equation (\ref{sehv}) implies
$$
b(e_h, \sigma) = \ell_u(\sigma) - s(\epsilon_h, \sigma).
$$
It follows that
\begin{equation}\label{EQ:April7:006}
\begin{split}
|b(e_h, \sigma)| & \leq |\ell_u(\sigma)| + \3bar \epsilon_h\3bar_w  \3bar \sigma\3bar_w\\
& \lesssim h^{k-1}(\|u\|_{k-1}+h\delta_{k,2}\|u\|_2) \3bar\sigma \3bar_w.
\end{split}
\end{equation}
Combing (\ref{EQ:April7:005}) with (\ref{EQ:April7:006}) gives rise to the following error estimate
$$
\frac12 \|e_h\|^2 \lesssim h^{k-1}(\|u\|_{k-1}+h\delta_{k,2}\|u\|_2) \|e_h\|,
$$
which is
\begin{equation}\label{EQ:April7:008}
\|e_h\| \lesssim h^{k-1}(\|u\|_{k-1}+h\delta_{k,2}\|u\|_2).
\end{equation}
The desired error estimate (\ref{erres}) is a direct result of (\ref{EQ:April7:002}) and
(\ref{EQ:April7:008}). This completes the proof of the theorem.
\end{proof}

In Table \ref{Table:rate-of-convergence}, we provide a summary for the rate of convergence for the numerical approximation $u_h$ arising from the WG scheme (\ref{32})-(\ref{2}). The first line of the table indicates the type of elements used in the numerical scheme. Recall that the space $W_{h,s}$ was employed for approximating $u_h$, while $V_{h,k}$ was for the auxiliary variable $\rho_h$. Although the solution $u_h$ is the quantity of major interest in this application, we believe that the auxiliary variable $\rho_h$ might provide some useful information for the design of error estimators for $u_h$. The second row of the table shows an optimal order of convergence for $u_h$ in the usual $L^2$-norm; i.e., a convergence of order $k$ when piecewise polynomials of degree $k-1$ are used.
\begin{table}[H]
\begin{center}
\caption{Convergence for the primal-dual weak Galerkin finite element method}\label{Table:rate-of-convergence}
\begin{tabular}{|c|c|c|}
\hline
        & $ V_{h,k}\times W_{h,k-2}, \ k\ge 2$ & $V_{h,k}\times W_{h, k-1},\ k\ge 1$  \\
\hline
$\|u-u_h\|$   & $h^{k-1} $ & $h^k$ \\
\hline
\end{tabular}
\end{center}
\end{table}

\subsection{Extension to $C^0$-type elements} By $C^0$-type elements, we mean a special class of finite element spaces $V_{h,k}$ consisting of weak finite element functions $v=\{v_0, v_b, \bv_g\}$ where $v_b = v_0|_\pT$ on each element $T\in \T_h$. Analogously, $C^{-1}$-type elements refer to the general case of $v=\{v_0, v_b, \bv_g\} \in V_{h,k}$ for which $v_b$ is totally independent of $v_0$. It is clear that $C^0$-type finite element schemes involve less number of
degrees of freedom than $C^{-1}$-type, as the boundary component $v_b$ can be obviously eliminated from the list of unknowns. However, $C^0$-type elements would impose more limitations on the geometry of the finite element partition $\T_h$.

The error estimates shown as in Theorem \ref{theoestimate} can be extended to $C^0$-type triangular elements for $V_{h,k}$. The rest of this section shall explain some modifications necessary for such an extension. First of all, for $C^0$-type elements, the discrete weak second order partial derivative $\partial^2_{ij,w} v$ can be computed as a polynomial in $P_s(T)$ on each element $T$ by the following equation
\begin{equation}\label{EQ:weak-partial-ij}
\begin{split}
(\partial^2_{ij, w}v,\varphi)_T=-(\partial_i
v_0,\partial_{j}\varphi)_T+ \langle v_{gi} ,\varphi
n_j\rangle_{\partial T},\qquad \forall \varphi\in P_s(T).
\end{split}
\end{equation}
For the convergence analysis to work, we need to have the error equations (\ref{sehv})-(\ref{sehv2}) which in turn require the commutative property (\ref{l})-(\ref{l-2}) for a properly defined projection operator $Q_h$ given as in (\ref{EQ:OperatorQh}). As there is no change on the variable $\bv_g$ and the space it lives in, the operator $\textbf{Q}_g$ should remain unchanged as the usual $L^2$ projection into the space of polynomials of degree $k-1$ on each piece of $\pT$. However, the operator $Q_0$ must be modified by using the interpolation operator $\tilde I_k$ given as in Lemma A.3 of \cite{gr}. For 2D triangular elements, this interplant polynomial $\tilde I_k v \in P_k(T)$ satisfies
\begin{eqnarray}
\int_e (v - \tilde I_k v) \phi ds &=& 0 \quad \forall \ \phi \in
P_{k-2}(e),\quad \forall \mbox{ side $e$ of $T$} \label{EQ:edge-moment}\\
\int_T (v - \tilde I_k v) \phi ds &=& 0 \quad \forall \ \phi \in
P_{k-3}(T).\label{EQ:element-moment}
\end{eqnarray}
From the integration by parts, (\ref{EQ:edge-moment}), and
(\ref{EQ:element-moment}), we then obtain
\begin{equation}\label{EQ:weak-derivative-100}
\begin{split}
(\partial_i \tilde I_k v,\partial_{j}\varphi)_T & = -( \tilde I_k v,
\partial_{ji}^2\varphi)_T + \langle \tilde I_k v,
\partial_{j}\varphi n_i\rangle_\pT \\
& = -( v,
\partial_{ji}^2\varphi)_T + \langle v,
\partial_{j}\varphi n_i\rangle_\pT \\
& = (\partial_i v,\partial_{j}\varphi)_T
\end{split}
\end{equation}
for all $\varphi\in P_{k-1}(T)$. Thus, from
(\ref{EQ:weak-partial-ij}), (\ref{EQ:weak-derivative-100}), and the fact that $s\le k-1$ we arrive at
\begin{equation*}
\begin{split}
(\partial^2_{ij, w} Q_hv,\varphi)_T &=-(\partial_i \tilde I_k
v,\partial_{j}\varphi)_T+ \langle (\textbf{Q}_g \nabla v)_i ,\varphi
n_j\rangle_{\partial T}\\
&= -(\partial_i v,\partial_{j}\varphi)_T+ \langle (\nabla v)_i,\varphi n_j\rangle_{\partial T}\\
& = (\partial^2_{ij} v,\varphi)_T \\
& = ({\cal Q}_h \partial^2_{ij} v,\varphi)_T
\end{split}
\end{equation*}
for all $\varphi\in P_{s}(T)$, which implies the commutative
property (\ref{l}). Note that (\ref{l-2}) clearly holds true as the weak gradient is identical with the strong gradient for $C^0$-type elements. Readers are encouraged to check out
\cite{mwyz-biharmonic} for a detailed discussion on the use of
$C^0$-type elements in the context of weak Galerkin finite element methods.

\section{Error Estimates for $L^2$ Projections}\label{Section:Estimates4L2Projections}
Recall that ${\cal T}_h$ is a shape-regular finite element
partition of the domain $\Omega$. For any $T\in {\cal T}_h$ and
$\phi\in H^1(T)$, the following trace inequality holds true
\cite{wy3655}:
\begin{equation}\label{tracein}
 \|\phi\|^2_{\partial T} \lesssim h_T^{-1}\|\phi\|_T^2+h_T \|\nabla \phi\|_T^2.
\end{equation}
If $\phi$ is a polynomial on the element $T\in {\cal T}_h$, then
from the inverse inequality, we have \cite{wy3655},
\begin{equation}
 \|\phi\|^2_{\partial T} \lesssim h_T^{-1}\|\phi\|_T^2.
\end{equation}

\begin{lemma}\cite{wy3655}
 Let ${\cal T}_h$ be a finite element partition of $\Omega$ satisfying the shape regular assumption given in \cite{wy3655}. For $0\leq t \leq \min(2,k)$, the following estimates hold true:
\begin{eqnarray}\label{error1}
& \sum_{T\in {\cal T}_h}h_T^{2t}\|u-Q_0u\|^2_{t,T} \lesssim
h^{2(m+1)}\|u\|^2_{m+1},&\qquad m\in [t-1,k],\ k\ge 1,\\
\label{error2}
& \sum_{T\in {\cal T}_h}h_T^{2t}\|u-{\cal Q}^{(k-1)}_hu\|^2_{t,T} \lesssim h^{2m}\|u\|^2_{m},&\qquad m\in [t,k],\ k\ge 1, \\
\label{error3} & \sum_{T\in {\cal T}_h}h_T^{2t}\|u-{\cal
Q}^{(k-2)}_hu\|^2_{t,T}  \lesssim h^{2m}\|u\|^2_{m},&\qquad  m\in
[t,k-1],\ k\ge 2.
\end{eqnarray}
Note that (\ref{error3}) is merely a different form of
(\ref{error2}).
 \end{lemma}

\begin{lemma} Assume that the coefficient tensor $a(x)$ and the drift vector $\bmu$ are uniformly piecewise smooth up to order $m-1$ in $\Omega$ with respect to the finite element partition ${\cal T}_h$. Then for any $v\in H^{m-1}(\Omega)\cap H^2(\Omega)$, the following estimates hold true:
\begin{eqnarray}\label{term1}
 & \Big(\sum_{T\in {\cal T}_h}h_T^{3}\| (\bmu
v-{\cal Q}^ {(k-1)}_h(\bmu {\cal Q}^ {(s)}_h v))\cdot
\bn\|^2_{\partial T}\Big)^{\frac{1}{2}} \lesssim h^{m}\|v\|_{m-1};\\
\label{term2} & \Big(\sum_{T\in {\cal T}_h}\sum_{i,j=1}^dh_T^{
3}\|\partial_i (a_{ij}v-{\cal Q}_h^{(s)}(a_{ij}{\cal Q}^ {(s)}_h
v))
\|^2_{\partial T}\Big)^{\frac{1}{2}} \\
& \quad \lesssim  h^{m-1}(\|v\|_{m-1}+h\delta_{m,2}\|v\|_2);\nonumber\\
\label{term3} & \Big(\sum_{T\in {\cal T}_h}\sum_{i,j=1}^dh_T \|
a_{ij}u-{\cal Q}_h^{(s)}(a_{ij}{\cal Q}^ {(s)}_hu)\|^2_{\partial
T}\Big)^{\frac{1}{2}} \lesssim h^{m-1} \|u\|_{m-1}.
\end{eqnarray}
Here, $m$ is an integer satisfying $m \in [2,k+1]$ if $s=k-1$ and
$m \in [2, k]$ if $s=k-2$, and $\delta_{m,2}$ is the usual
Kronecker's delta with value $1$ when $m=2$ and value $0$
otherwise.
\end{lemma}

\begin{proof}
To prove (\ref{term1}), from the trace inequality (\ref{tracein})
and the estimate (\ref{error3}) we have
 \begin{equation*}
  \begin{split}
   & \sum_{T\in {\cal T}_h}h_T^{ 3}\| (\bmu
v-{\cal Q}^ {(k-1)}_h(\bmu {\cal Q}^ {(s)}_h v))\cdot
\bn\|^2_{\partial T}\\
\lesssim & \sum_{T\in {\cal T}_h}h_T^{ 2}\| \bmu v-{\cal Q}^
{(k-1)}_h(\bmu {\cal Q}^ {(s)}_h v)\|^2_{ T}+h_T^{4}\| \bmu
v-{\cal Q}^ {(k-1)}_h(\bmu {\cal Q}^ {(s)}_h v)\|^2_{ 1,T}\\
\lesssim & \sum_{T\in {\cal T}_h} h_T^{ 2}\|  \bmu v-{\cal Q}^
{(k-1)}_h(\bmu  v) \|^2_{ T}+ h_T^{ 2}\| {\cal Q}^ {(k-1)}_h(\bmu
v)-{\cal Q}^ {(k-1)}_h(\bmu {\cal Q}^ {(s)}_h v) \|^2_{ T}
\\&\qquad+h_T^{4}\|  \bmu
v-{\cal Q}^ {(k-1)}_h(\bmu v) \|^2_{ 1,T}+h_T^{4}\| {\cal Q}^
{(k-1)}_h(\bmu v)
-{\cal Q}^ {(k-1)}_h(\bmu {\cal Q}^ {(s)}_h v) \|^2_{ 1,T}\\
 \lesssim & \sum_{T\in {\cal T}_h}
h_T^{ 2}\|  \bmu v-{\cal Q}^ {(k-1)}_h(\bmu v) \|^2_{ T}+ h_T^{
2}\|  \bmu  v - \bmu {\cal Q}^ {(s)}_h v  \|^2_{ T}
\\&\qquad+h_T^{4}\|  \bmu v-{\cal Q}^ {(k-1)}_h(\bmu v) \|^2_{ 1,T}+h_T^{4}\| \bmu v
- \bmu {\cal Q}^ {(s)}_h v  \|^2_{ 1,T}\\
\lesssim & \ h^{2m}\|v\|^2_{m-1}
  \end{split}
 \end{equation*}
 for $m\in [2, k+1]$ when $s=k-1$ and $m\in [2, k]$ when $s=k-2$.

As to (\ref{term2}), we use the trace inequality (\ref{tracein})
and the estimate (\ref{error3}) to obtain
 \begin{equation*}
  \begin{split}
&\sum_{T\in {\cal T}_h}\sum_{i,j=1}^dh_T^{ 3}\|\partial_i
(a_{ij}v-{\cal Q}_h^{(s)}(a_{ij}{\cal Q}^ {(s)}_h v))
\|^2_{\partial T}\\
 \lesssim & \sum_{T\in {\cal T}_h}\sum_{i,j=1}^dh_T^{ 2}\|\partial_i
(a_{ij}v-{\cal Q}_h^{(s)}(a_{ij}{\cal Q}^ {(s)}_h v)) \|^2_{
T}+h_T^{ 4}\|\partial_i (a_{ij}v-{\cal Q}_h^{(s)}(a_{ij}{\cal Q}^
{(s)}_h v))
\|^2_{ 1, T}\\
 \lesssim & \sum_{T\in {\cal T}_h}\sum_{i,j=1}^dh_T^{ 2}\|\partial_i
(a_{ij}v-{\cal Q}_h^{(s)}(a_{ij} v)) \|^2_{ T}+h_T^{
2}\|\partial_i ({\cal Q}_h^{(s)}(a_{ij} v)-{\cal
Q}_h^{(s)}(a_{ij}{\cal Q}^ {(s)}_h v)) \|^2_{ T}
\\&\qquad +h_T^{ 4}\|\partial_i
(a_{ij}v-{\cal Q}_h^{(s)}(a_{ij} v)) \|^2_{ 1, T}+h_T^{
4}\|\partial_i ({\cal Q}_h^{(s)}(a_{ij} v)-{\cal
Q}_h^{(s)}(a_{ij}{\cal Q}^ {(s)}_h v))
\|^2_{ 1, T}\\
\lesssim  & \ h^{2m-2}(\|v\|^2_{m-1}+h^2\delta_{m,2}\|v\|^2_2),
  \end{split}
 \end{equation*}
 for $m\in [2, k+1]$ when $s=k-1$ and $m\in [2, k]$ when $s=k-2$.

Finally for (\ref{term3}), we again use the trace inequality
(\ref{tracein}) and the estimate (\ref{error3}) to obtain
\begin{equation*}
\begin{split}
 &\sum_{T\in {\cal T}_h}\sum_{i,j=1}^dh_T
\| a_{ij}v-{\cal Q}_h^{(s)}(a_{ij}{\cal Q}^ {(s)}_h v)\|^2_{\partial T} \\
 \lesssim & \sum_{T\in {\cal T}_h}\sum_{i,j=1}^d
\| a_{ij}v-{\cal Q}_h^{(s)}(a_{ij}{\cal Q}^{(s)}_h v)\|^2_{
T}  +h_T^2\| a_{ij}v-{\cal Q}_h^{(s)}(a_{ij}{\cal Q}^ {(s)}_h v)\|^2_{1,T}  \\
 \lesssim & \sum_{T\in {\cal T}_h}\sum_{i,j=1}^d
\| a_{ij}v-{\cal Q}_h^{(s)}(a_{ij} v)\|^2_{ T} +\| {\cal
Q}_h^{(s)}(a_{ij} v)-{\cal Q}_h^{(s)}(a_{ij}{\cal Q}^ {(s)}_h
v)\|^2_{
T} \\
&\qquad +h_T^2\| a_{ij} v-{\cal Q}_h^{(s)}(a_{ij} v)\|^2_{1,T} +
h_T^2\|{\cal Q}_h^{(s)}(a_{ij} v)-{\cal Q}_h^{(s)}(a_{ij}{\cal Q}^{(s)}_h v)\|^2_{1,T} \\
 \lesssim & \sum_{T\in {\cal T}_h}\sum_{i,j=1}^d
\| a_{ij}v-{\cal Q}_h^{(s)}(a_{ij} v)\|^2_{ T} +\|a_{ij} v -
a_{ij}{\cal Q}^ {(s)}_h v \|^2_{
T} \\
&\qquad +h_T^2\| a_{ij}v-{\cal Q}_h^{(s)}(a_{ij} v)\|^2_{1,T} +
h_T^2\| a_{ij} v - a_{ij}{\cal Q}^ {(s)}_h v \|^2_{1,T} \\
 \lesssim & \ h^{2m-2}\|v\|^2_{m-1},
 \end{split}
 \end{equation*}
 where $m\in [2, k+1]$ if $s=k-1$ and $m\in [2, k]$ if $s=k-2$. This completes the proof of the lemma.
\end{proof}

\section{Numerical Results}\label{Section:numerics}
The goal of this section is to present some numerical results for
the finite element scheme (\ref{32})-(\ref{2}). Our test problem
seeks $u$ satisfying
\begin{equation}\label{Test-Problem}
\begin{split}
\nabla \cdot (\bmu u)-\frac{1}{2}\sum_{i,j=1}^d
\partial^2_{ij}(a_{ij}u)=&f,\quad \text{in}\
\Omega,\\
u =& g,\quad \text{on}\ \partial\Omega,
\end{split}
\end{equation}
where $\Omega$ is a polygonal domain in 2D. As the problem
(\ref{Test-Problem}) has non-homogeneous Dirichlet data on the
boundary, the weak formulation (\ref{weakform}) must be modified
accordingly so that the weak solution $u=u(x)\in L^2(\Omega)$ is
given by the following equation
\begin{equation}\label{weakform-10}
\begin{split}
\int_\Omega u {\cal L}(v) dx =\frac{1}{2} \sum_{i,j=1}^d
\int_{\partial\Omega} g a_{ij} n_i \partial_j
 v ds - \int_\Omega f v dx, \qquad \forall v\in H^2(\Omega)\cap
 H_0^1(\Omega).
\end{split}
\end{equation}
The corresponding primal-dual weak Galerkin finite element scheme
seeks $(u_h;\rho_h)\in W_{h,s} \times V^0_{h,k}$ satisfying
 \begin{eqnarray}\label{32-new}
s( \rho_h, \sigma)+b(\sigma, u_h)&=&\frac{1}{2}
\sum_{i,j=1}^d\langle a_{ij}g, \sigma_{gj}n_i\rangle_{\partial
\Omega}-(f, \sigma_0)
,\qquad \forall\sigma\in V^0_{h,k},\\
b(\rho_h, v)&=&0,\qquad\ \qquad \qquad\qquad\qquad\qquad\qquad
\forall v\in W_{h,s}.\label{2-new}
\end{eqnarray}

Our numerical implementation is based on the PD-WG scheme
(\ref{32-new})-(\ref{2-new}) with the element of order $k=2$ on
uniformly triangular finite element partitions. This configuration
corresponds to the following selection for the finite element
spaces:
$$
V_{h,2}=\{\rho=\{\rho_0,\rho_b, \brho_g\}: \ \rho_0\in P_2(T), \rho_b\in P_2(e),
\brho_g\in [P_1(e)]^2, e\subset\pT, T\in {\cal T}_h\},
$$
and
$$
W_{h,s}=\{w: \ w|_T \in P_s(T),\ \forall T\in {\cal T}_h
\}, \quad s=0 \mbox{ or } 1.
$$

For simplicity, our numerical experiments will be conducted for only $C^0$-type elements for which $\rho_b$ is identical with the trace of $\rho_0$ on $\pT$ for any $T\in \T_h$. For convenience, the $C^0$-type WG element with $s=1$ (i.e., $W_{h,1}$) shall be named as the
$P_2(T)/[P_1(\pT)]^2/P_1(T)$ element. Analogously, the case corresponding to $s=0$ (i.e., $W_{h,0}$) is named as $P_2(T)/[P_1(\pT)]^2/P_0(T)$.

Two polygonal domains are considered in the numerical experiments, with the first one being
the unit square $\Omega=(0,1)^2$ and the second an L-shaped region with
vertices $A_0=(0,0), \ A_1=(2,0),\ A_2=(2,1), \ A_3=(1,1), \ A_4=(1,2)$, and
$A_5=(0,2)$. Note that the error estimates developed in the previous section are applicable to the square domain, but not to the L-shaped domain. Nevertheless, our primal-dual weak Galerkin finite element scheme is well formulated on any domains. The L-shaped domain is chosen just for the purposes of demonstrating the performance of the algorithm for cases for which theory has not been developed.

The finite element partitions in our computation are obtained through a simple uniform refinement procedure as follows. Given an initial coarse triangulation of the domain, a
sequence of triangulations are obtained successively through a uniform refinement that divides
each coarse level triangle into four congruent sub-triangles by
connecting the three mid-points on the edges. We use $\rho_h=\{\rho_0,
\brho_g\}\in V_{h,2}$ and $u_h\in W_{h,s}, \ s=0,1,$ to denote the numerical solutions
arising from (\ref{32})-(\ref{2}). The numerical solutions are compared with carefully chosen
interpolations of the exact solution in various norms. In particular, the primal variable $u_h$ is compared with the exact solution $u$ on each element at either
three vertices (for $s=1$) or the center (for $s=0$), which is known as the nodal point interpolation and is denoted as $I_h u$. The auxiliary variable $\rho_h$ is supposed to approximate the true solution $\rho=0$, and is therefore compared with $Q_h\rho =0$. The error functions are thus denoted as
$$
\epsilon_h=\rho_h-Q_h\rho\equiv \{\rho_0, \brho_g\}, \quad e_h=u_h- I_h u.
$$

The following norms are applied in the computation for the error:
\begin{eqnarray*}
\mbox{$L^2$- norm:}\quad & &  \3bar \rho_h\3bar_0=\Big(\sum_{T\in {\cal T}_h}
\int_T |\rho_0|^2 dT\Big)^{\frac{1}{2}},\\
\mbox{Semi $H^1$-norm:}\quad  & &
\3bar \brho_h\3bar_1=\Big(\sum_{T\in {\cal T}_h} h_T
\int_{\partial T}
|\brho_g|^2 ds\Big)^{\frac{1}{2}},\\
\mbox{$L^2$-norm:}\quad & &  \|e_h\|_0=\Big(\sum_{T\in {\cal
T}_h} \int_T |e_h|^2 dT\Big)^{\frac{1}{2}}.
\end{eqnarray*}

Tables \ref{NE:TRI:Case1-1}--\ref{NE:TRI:Case1-2} illustrate the
performance of $C^0$-type $P_2(T)/[P_1(\pT)]^2/P_1(T)$ element for the test problem
(\ref{Test-Problem}) with exact solution given by $u=\sin(x_1)\sin(x_2)$ on the unit square domain and the L-shaped domain. The right-hand side function and the Dirichlet boundary
condition are chosen to match the exact solution. Our numerical results
indicate that the convergence rate for the solution $u_h$ is of order $r=2$ in the
discrete $L^2$-norm on both the unit square domain and the L-shaped domain. The result is in great consistency with the theoretical rate of convergence for $u_h$.

\begin{table}[H]
\begin{center}
\caption{Numerical rates of convergence for the $C^0$-
$P_2(T)/[P_1(\pT)]^2/P_1(T)$ element applied to problem
(\ref{Test-Problem}) with exact solution $u=\sin(x_1)\sin(x_2)$ on
$\Omega=(0,1)^2$. The coefficient matrix is $a_{11}=3$,
$a_{12}=a_{21}=1$, and $a_{22}=2$. The drift vector is
$\mu=[1,1]$.}\label{NE:TRI:Case1-1}
\begin{tabular}{|c|c|c|c|c|c|c|}
\hline
$1/h$        & $\3bar\rho_h\3bar_0 $ & order &  $\3bar\brho_g\3bar_{1} $  & order  &   $\|u_h-I_h u\|_0$  & order  \\
\hline
1   &3.52e-01   &   & 4.02e-00 &&       2.11e-01 &  \\
\hline
2 &  4.09e-02   &3.11&  7.16e-01    &2.49 & 5.55e-02&   1.93
\\
\hline
4 & 3.77e-03&   3.44 &  1.22e-01 &  2.55    & 1.42e-02& 1.97
  \\
\hline
8 & 2.78e-04    &3.76&  1.79e-02&   2.79    & 3.57e-03& 1.99
\\
\hline
16 & 1.87e-05   &3.90&  2.36e-03    &2.91   & 8.93e-04  &2.00
 \\
\hline
32 & 1.20e-06&  3.95    &3.04e-04&  2.96&   2.23e-04    &2.00
\\
\hline
\end{tabular}
\end{center}
\end{table}

\begin{table}[H]
\begin{center}
\caption{Numerical rates of convergence for the $C^0$-
$P_2(T)/[P_1(\pT)]^2/P_1(T)$ element applied to the problem
(\ref{Test-Problem}) with exact solution $u=\sin(x_1)\sin(x_2)$ on
the L-shaped domain. The coefficient matrix is $a_{11}=3$,
$a_{12}=a_{21}=1$, and $a_{22}=2$. The drift vector is $\mu=[1,1]$.}\label{NE:TRI:Case1-2}
\begin{tabular}{|c|c|c|c|c|c|c|}
\hline
$1/h$        & $\3bar\rho_h\3bar_0 $ & order &  $\3bar\brho_g\3bar_{1} $  & order  &   $\|u_h-I_h u\|_0$  & order  \\
\hline
1    & 5.33e-01     & & 5.20e-00    &&  3.07e-01    & \\
\hline
2  & 7.77e-02 &2.78& 1.18e-00&  2.15 &  8.03e-02 &1.94 \\
\hline
4 & 6.42e-03&3.60   & 1.90e-01 &2.63&   2.04e-02    &1.98 \\
\hline
8 &  4.49e-04&  3.84& 2.65e-02  &2.84   &5.11e-03&  2.00 \\
\hline
16 & 2.95e-05 &3.93& 3.47e-03 & 2.93    &1.27e-03 & 2.00  \\
\hline
 32 & 1.89e-06&3.97& 4.44e-04 &2.97     &3.18e-04&2.00
 \\
\hline
\end{tabular}
\end{center}
\end{table}

Tables \ref{NE:TRI:Case5-1}--\ref{NE:TRI:Case5-2} illustrate the
performance of the $C^0$-type $P_2(T)/[P_1(\pT)]^2/P_0(T)$ element for the test problem
(\ref{Test-Problem}) with exact solution given by $u=\sin(x_1)\sin(x_2)$ on the unit square domain and the L-shaped domain with constant coefficients. The numerical results are in
good consistency with what the theory predicts.

\begin{table}[H]
\begin{center}
\caption{Numerical rates of convergence for the $C^0$-
$P_2(T)/[P_1(\pT)]^2/P_0(T)$ element applied to the problem
(\ref{Test-Problem}) with exact solution $u=\sin(x_1)\sin(x_2)$ on
$\Omega=(0,1)^2$. The coefficient matrix is $a_{11}=3$,
$a_{12}=a_{21}=1$, and $a_{22}=2$. The drift  vector is $\mu=[1,1]$.}\label{NE:TRI:Case5-1}
\begin{tabular}{|c|c|c|c|c|c|c|}
\hline
$1/h$        & $\3bar\rho_h\3bar_0 $ & order &  $\3bar\brho_g\3bar_{1} $  & order  &   $\|u_h-I_h u\|_0$  & order  \\
\hline
1   & 9.82e-03  &&  2.57e-01    &&  1.03e-01&
 \\
\hline
2 & 4.79e-03    &1.04&  6.98e-02 &1.88  &   4.38e-02  & 1.23
\\
\hline
4 & 1.53e-03    &1.65&  1.81e-02 &1.95& 1.72e-02&   1.35
 \\
\hline
8 & 3.76e-04&   2.02 &  4.43e-03 &  2.03 & 7.74e-03 &1.15
\\
\hline
16 & 9.27e-05   &2.02   & 1.09e-03  &2.02   & 3.77e-03  &1.04
 \\
\hline
32 &2.31e-05    &2.01&  2.72e-04&   2.01& 1.87e-03  &1.01
\\
\hline
\end{tabular}
\end{center}
\end{table}

\begin{table}[H]
\begin{center}
\caption{Convergence rates for the $C^0$-
$P_2(T)/[P_1(\pT)]^2/P_0(T)$ element applied to problem
(\ref{Test-Problem}) with exact solution $u=\sin(x_1)\sin(x_2)$ on
the L-shaped domain. The coefficient matrix is $a_{11}=3$,
$a_{12}=a_{21}=1$, and $a_{22}=2$. The drift  vector is $\mu=[1,1]$.}\label{NE:TRI:Case5-2}
\begin{tabular}{|c|c|c|c|c|c|c|}
\hline
$1/h$        & $\3bar\rho_h\3bar_0 $ & order &  $\3bar\brho_g\3bar_{1} $  & order  &   $\|u_h-I_h u\|_0$  & order  \\
\hline
1  &  0.0199    &&  5.01e-01    && 1.72e-01 & \\
\hline
2  &0.0165 &    0.269   &1.39e-01 & 1.85& 7.65e-02 &    1.17
 \\
\hline
4 &0.00470  &1.81   & 3.60e-02& 1.95    & 3.37e-02  &1.18
 \\
\hline
8 & 1.18E-03    &2.00 &  8.96e-03 & 2.01  & 1.61e-02    &1.07
 \\
\hline
16 &2.93E-04&   2.01&   2.23e-03&   2.01&   7.92e-03    &1.02
 \\
\hline
 32 & 7.32E-05& 2.00 & 5.56e-04 &2.00 & 3.94e-03 & 1.01
\\
\hline
\end{tabular}
\end{center}
\end{table}

Tables \ref{NE:TRI:Case7-1}--\ref{NE:TRI:Case7-2} illustrate the
performance of the $C^0$-type $P_2(T)/[P_1(\pT)]^2/P_0(T)$ element for the test problem
(\ref{Test-Problem}) on the unit square domain and the L-shaped domain. The exact solution is given by $u=\sin(x_1)\sin(x_2)$ and the differential operator has variable coefficients. The numerical rate of convergence for the primal-dual WG finite element method is consistent with the theory.

\begin{table}[H]
\begin{center}
\caption{Numerical rates of convergence for the $C^0$-
$P_2(T)/[P_1(\pT)]^2/P_0(T)$ element applied to problem
(\ref{Test-Problem}) with exact solution $u=\sin(x_1)\sin(x_2)$ on
$\Omega=(0,1)^2$. The coefficient matrix is given by $a_{11}=1+x_1^2$,
$a_{12}=a_{21}=0.25x_1x_2$, and $a_{22}=1+x_2^2$. The drift vector is $\mu=[x,y]$.}\label{NE:TRI:Case7-1}
\begin{tabular}{|c|c|c|c|c|c|c|}
\hline
$1/h$        & $\3bar\rho_h\3bar_0 $ & order &  $\3bar\brho_g\3bar_{1} $  & order  &   $\|u_h-I_h u\|_0$  & order  \\
\hline
1   & 1.18e-02  &&  1.94e-01    &&  4.88e-02 &
  \\
\hline
2 &  1.09e-03   &3.44&  4.70e-02    &2.05&  2.49e-02    &0.97
\\
\hline
4 & 3.13e-04    &1.81&  1.13e-02 &  2.05    & 1.14e-02  &1.13
  \\
\hline
8 & 8.76e-05    &1.83 & 2.77e-03    &2.03   & 5.52e-03 &1.04
\\
\hline
16 & 2.25e-05   &1.96&  6.85e-04&   2.02&   2.74e-03 &1.01
 \\
\hline
32 & 5.67e-06   &2.00 & 1.70e-04    &2.01   & 1.37e-03  &1.00
\\
\hline
\end{tabular}
\end{center}
\end{table}

\begin{table}[H]
\begin{center}
\caption{Numerical rates of convergence rates for the $C^0$-
$P_2(T)/[P_1(\pT)]^2/P_0(T)$ element applied to the problem
(\ref{Test-Problem}) with exact solution $u=\sin(x_1)\sin(x_2)$ on
the L-shaped domain. The coefficient matrix is given by $a_{11}=1+x_1^2$,
$a_{12}=a_{21}=0.25x_1x_2$, and $a_{22}=1+x_2^2$. The drift vector is $\mu=[x,y]$.}\label{NE:TRI:Case7-2}
\begin{tabular}{|c|c|c|c|c|c|c|}
\hline
$1/h$        & $\3bar\rho_h\3bar_0 $ & order &  $\3bar\brho_g\3bar_{1} $  & order  &   $\|u_h-I_h u\|_0$  & order  \\
\hline
1    & 2.12e-02     && 5.39e-01 &&  2.03e-01    & \\
\hline
2  & 1.25e-02   &0.763  & 1.33e-01 &    2.02    & 9.15e-02 &    1.15
 \\
\hline
4 & 3.09e-03    &2.01   & 3.24e-02  &2.04 &  4.23e-02   &1.11\\
\hline
8 & 7.60e-04&   2.02 &  7.93e-03 &2.03  & 2.06e-02  &1.04
\\
\hline
16 & 1.89e-04   &2.01   &1.96e-03&  2.02 &  1.02e-02    &1.01
  \\
\hline
 32 & 4.71e-05& 2.00 &4.87e-04&2.01 & 5.10e-03 &1.00
 \\
\hline
\end{tabular}
\end{center}
\end{table}

Tables \ref{NE:TRI:Case9-1}--\ref{NE:TRI:Case9-3} illustrate the
performance of the $C^0$-type $P_2(T)/[P_1(\pT)]^2/P_1(T)$ element for the test problem
(\ref{Test-Problem}) when the parameter $\delta$ varies in the stabilizer $s(\cdot,\cdot)$. The exact solution is given by $u=\sin(x_1)\sin(x_2)$, and the domain in this test case is the unit square with constant coefficients and drifting. The results
indicate that the convergence rate for the solution $u_h$ of the weak
Galerkin algorithm (\ref{32})-(\ref{2}) is of order $r=2$ in the
discrete $L^2$-norm for $u_h$ for different values of $\delta$. Table \ref{NE:TRI:Case9-6} illustrates the performance of the numerical scheme on the L-shaped domain with $\delta=10000$.
It is interesting to note that the absolute error decreases as $\delta$ increases.

\begin{table}[H]
\begin{center}
\caption{Numerical rates of convergence for the $C^0$-
$P_2(T)/[P_1(\pT)]^2/P_1(T)$ element applied to the problem
(\ref{Test-Problem}) with exact solution $u=\sin(x_1)\sin(x_2)$ on
$\Omega=(0,1)^2$. The coefficient matrix is $a_{11}=3$,
$a_{12}=a_{21}=1$, and $a_{22}=2$. The drift vector is $\mu=[1,1]$. The stabilizer parameter has value $\delta=0.1$.}\label{NE:TRI:Case9-1}
\begin{tabular}{|c|c|c|c|c|c|c|}
\hline
$1/h$        & $\3bar\rho_h\3bar_0 $ & order &  $\3bar\brho_g\3bar_{1} $  & order  &   $\|u_h-I_h u\|_0$  & order  \\
\hline
1   &2.22E-01&& 2.59E+00&&1.40E-01 &    \\
\hline
2 &   2.12E-02&3.39&4.51E-01&2.52 &4.58E-02&1.61
\\
\hline
4 & 1.62E-03&3.71 &6.74E-02&2.74 &1.29E-02&1.83
  \\
\hline
8 & 1.09E-04&3.89&8.97E-03&2.91&3.39E-03&1.93
\\
\hline
16 & 7.01E-06&3.96 &1.15E-03&2.97&8.65E-04&1.97
 \\
\hline
32 &  4.41E-07&3.99&1.44E-04&2.99&2.18E-04&1.99
\\
\hline
\end{tabular}
\end{center}
\end{table}

\begin{table}[H]
\begin{center}
\caption{Numerical rates of convergence for the $C^0$-
$P_2(T)/[P_1(\pT)]^2/P_1(T)$ element applied to the problem
(\ref{Test-Problem}) with exact solution $u=\sin(x_1)\sin(x_2)$ on
$\Omega=(0,1)^2$. The coefficient matrix is $a_{11}=3$,
$a_{12}=a_{21}=1$, and $a_{22}=2$. The drift vector is $\mu=[1,1]$. The stabilizer parameter has value $\delta=1.0$.}\label{NE:TRI:Case9-2}
\begin{tabular}{|c|c|c|c|c|c|c|}
\hline
$1/h$        & $\3bar\rho_h\3bar_0 $ & order &  $\3bar\brho_g\3bar_{1} $  & order  &   $\|u_h-I_h u\|_0$  & order  \\
\hline
1   & 5.11e-02 && 8.90e-01 && 5.58e-02 &\\
\hline
2 & 8.51e-03 &2.59& 1.83e-01&2.28 & 2.20e-02 &1.35
\\
\hline
4 & 1.03e-03 &3.04 & 4.41e-02&2.05 & 8.68e-03&1.34
 \\
\hline
8 & 9.02e-05&3.52& 7.65e-03&2.53& 2.88e-03&1.59
\\
\hline
16 & 6.44e-06&3.81& 1.08e-03 &2.83&8.12e-04&1.82
 \\
\hline
32 & 4.24e-07&3.93&1.40e-04&2.94 &2.13e-04&1.93
\\
\hline
\end{tabular}
\end{center}
\end{table}

\begin{table}[H]
\begin{center}
\caption{Numerical rates of convergence for the $C^0$-
$P_2(T)/[P_1(\pT)]^2/P_1(T)$ element applied to the problem
(\ref{Test-Problem}) with exact solution $u=\sin(x_1)\sin(x_2)$ on
$\Omega=(0,1)^2$. The coefficient matrix is $a_{11}=3$,
$a_{12}=a_{21}=1$, and $a_{22}=2$. The drift vector is $\mu=[1,1]$. The stabilizer parameter has value $\delta=10000$.}\label{NE:TRI:Case9-3}
\begin{tabular}{|c|c|c|c|c|c|c|}
\hline
$1/h$        & $\3bar\rho_h\3bar_0 $ & order &  $\3bar\brho_g\3bar_{1} $  & order  &   $\|u_h-I_h u\|_0$  & order  \\
\hline
1   & 5.98e-06&&6.78e-01&& 4.04e-02&\\
\hline
2 & 1.35e-06&2.15&3.08e-02&4.46 &9.48e-03&2.09
\\
\hline
4 & 6.40e-07&1.08&1.39e-03&4.47&2.05e-03&2.21
 \\
\hline
8 & 2.26e-07&1.50 &6.30e-05&4.46 &4.85e-04&2.08
\\
\hline
16 & 5.97e-08&1.92 &6.38e-06&3.30 &1.21e-04&2.01
 \\
\hline
32 &  1.29e-08&2.20 &2.60e-06&1.30&3.15e-05& 1.94
\\
\hline
\end{tabular}
\end{center}
\end{table}

\begin{table}[H]
\begin{center}
\caption{Numerical rates of convergence for the $C^0$-
$P_2(T)/[P_1(\pT)]^2/P_1(T)$ element applied to the problem
(\ref{Test-Problem}) with exact solution $u=\sin(x_1)\sin(x_2)$ on the L-shaped domain. The coefficient matrix is $a_{11}=3$, $a_{12}=a_{21}=1$, and $a_{22}=2$. The drift vector is $\mu=[1,1]$. The stabilizer parameter has value $\delta=10000$.}\label{NE:TRI:Case9-6}
\begin{tabular}{|c|c|c|c|c|c|c|}
\hline
$1/h$        & $\3bar\rho_h\3bar_0 $ & order &  $\3bar\brho_g\3bar_{1} $  & order  &   $\|u_h-I_h u\|_0$  & order  \\
\hline
 1 &  8.31e-06&&8.36e-01&&1.33e-01&
 \\
\hline
  2 & 6.90e-06&0.269&3.64e-02&4.52 &2.51e-02&2.40
 \\
\hline
4 & 2.77e-06&1.32&1.56e-03&4.54&5.88e-03&2.09
 \\
\hline
8 & 7.97e-07&1.79 &8.03e-05&4.28 &1.45e-03&2.03
 \\
\hline
 16 & 1.99e-07&2.00&2.11e-05&1.93&3.59e-04&2.01
 \\
\hline
 32&4.35e-08&2.20&9.34e-06&1.18&8.99e-05&2.00
 \\
\hline
\end{tabular}
\end{center}
\end{table}

Figure \ref{Results-100} illustrates the performance of our
numerical methods for a problem with discontinuous solution. This
test problem has the following configuration: domain
$\Omega=(-1,1)\times (-1,1)$, diffusion tensor $a=\alpha I$ with
$\alpha=1$ for $x_1<0$ and $\alpha=2$ for $x_1\ge 0$, load
function $f=0$, Dirichlet boundary condition $g=2$ for $x_1<0$ and
$g=1$ for $x_1>0$. The numerical solutions are obtained by using
piecewise linear functions for the primal variable; i.e., $s=1$ in
the finite element space $W_{h,s}$. For this test problem, the
exact solution is known to be $u=2$ for $x_1<0$ and $u=1$ for
$x_1>0$ when the drift vector is zero. The plot indicates that the
primal-dual weak Galerkin finite element solution is very
consistent with the exact solution when $\bmu=0$ (see left plot).
The right plot corresponds to the case of $\bmu=[1,1]$ for which
no exact solution is known.

\begin{figure}[h]
\centering
\begin{tabular}{cc}
  \resizebox{2.4in}{2.1in}{\includegraphics{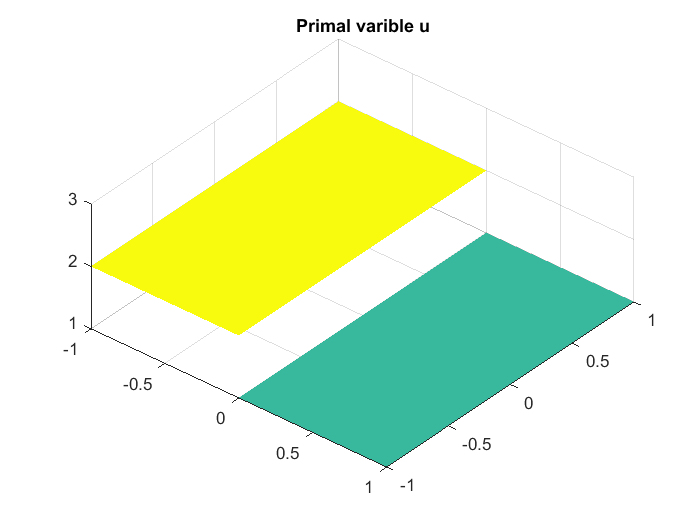}}
  \resizebox{2.4in}{2.1in}{\includegraphics{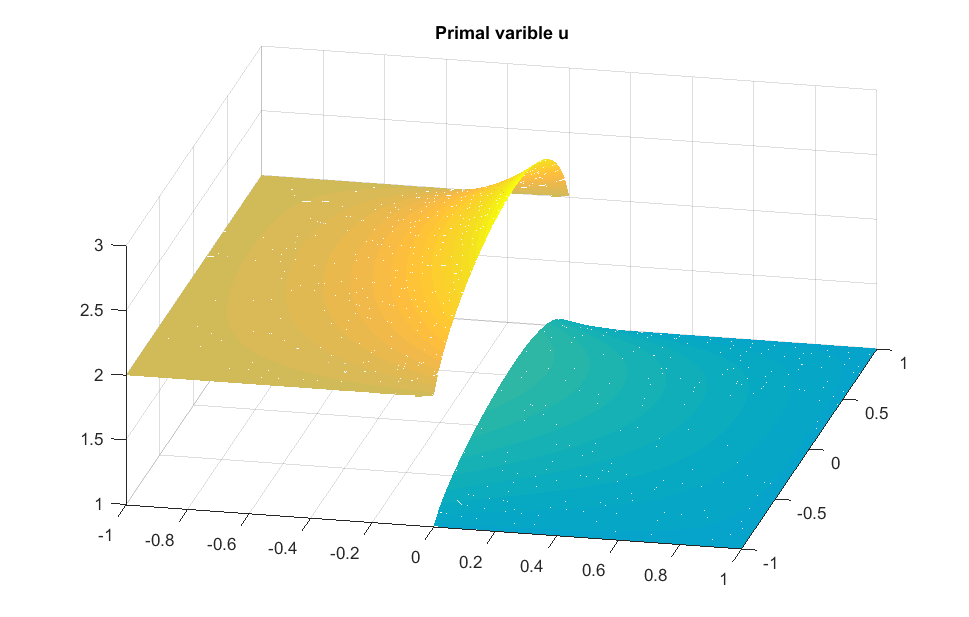}}
\end{tabular}
  \caption{Domain $\Omega=(-1,1)\times (-1,1)$, diffusion tensor $a=\alpha I$ with $\alpha=1$ for $x_1<0$ and $\alpha=2$ for $x_1\ge 0$,
  load function $f=0$, Dirichlet boundary data $g=2$ for $x_1<0$ and $g=1$ for $x_1>0$. Plot for primal variable $u_h$: drift vector $\bmu=0$ (left),
  drift vector $\bmu=[1,1]$ (right).}
  \label{Results-100}
\end{figure}

The primal-dual weak Galerkin finite element method
(\ref{32})-(\ref{2}) was further applied to a test problem for
which the exact solution is discontinuous along the $x_2$-axis.
The configuration of this test problem is given as follows: domain
$\Omega=(-1,1)\times (-1,1)$, diffusion tensor $a=\alpha I$ with
$\alpha=1$ for $x_1<0$ and $\alpha=2$ for $x_1>0$, load function
$f=9\sin(3x_2)$, Dirichlet boudnary data $g=2\sin(3x_2)$ for
$x_1<0$ and $g=\sin(3x_2)$ for $x_1>0$, and drift vector $\bmu=0$.
We use piecewise linear functions to approximate the primal
variable, and the profile of the corresponding numerical solution
is presented in Figure \ref{Results-101}. The exact solution for
this test problem is given by $u=2\sin(3x_2)$ for $x_1<0$ and
$u=\sin(3x_2)$ for $x_1>0$. The plot shows a great consistency
between the numerical solution and the exact one.

\begin{figure}[h]
\centering
\begin{tabular}{cc}
  \resizebox{2.4in}{1.8in}{\includegraphics{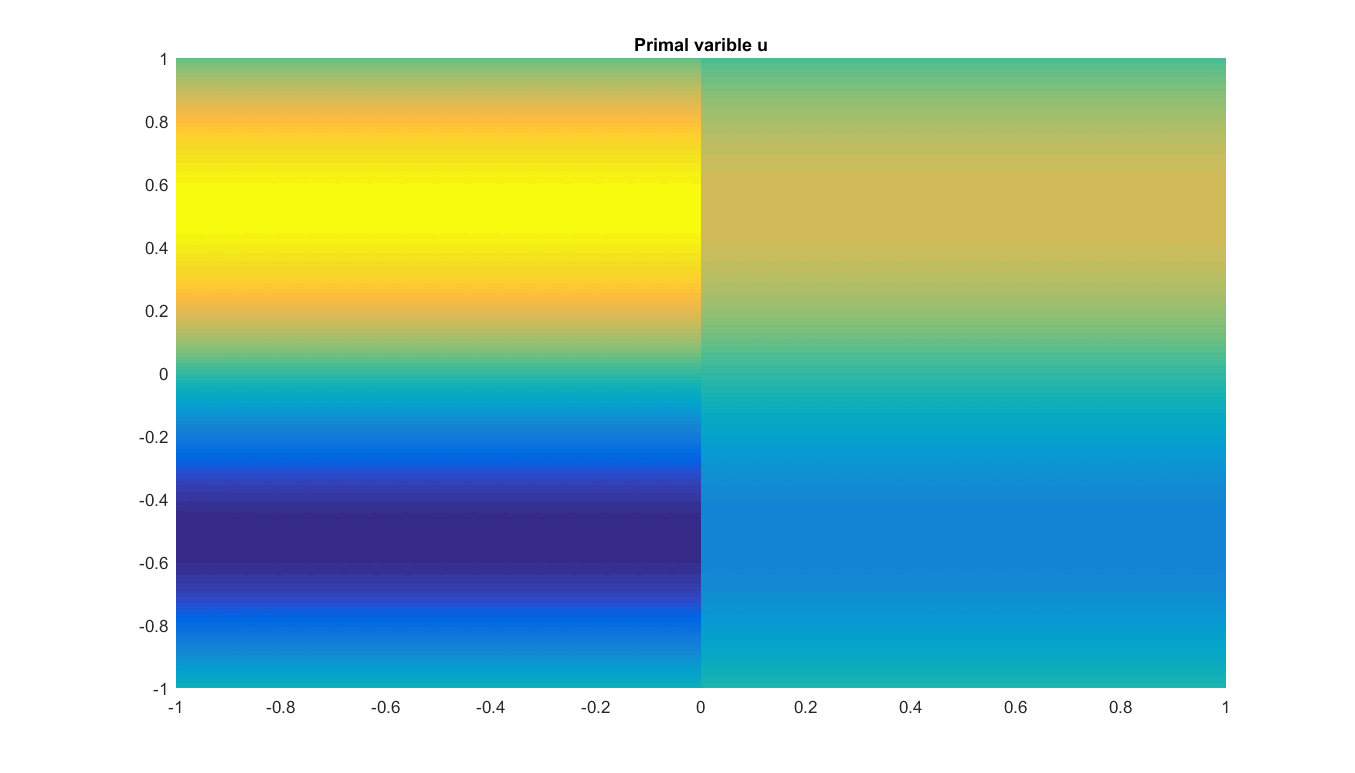}}
  \resizebox{2.4in}{2.1in}{\includegraphics{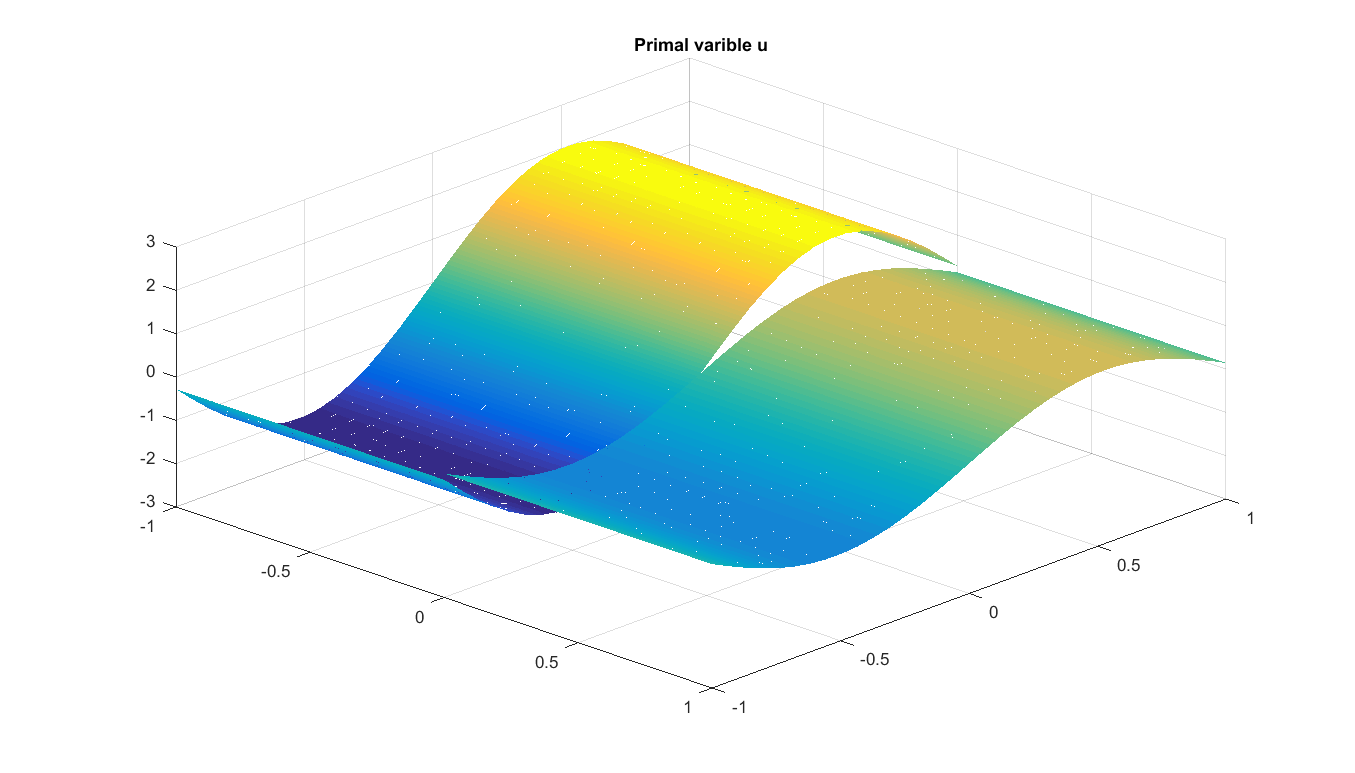}}
\end{tabular}
  \caption{Domain $\Omega=(-1,1)\times (-1,1)$, diffusion tensor $a=\alpha I$ with $\alpha=1$ for $x_1<0$ and $\alpha=2$ for $x_1>0$,
  load function $f=9\sin(3x_2)$, Dirichlet boudnary data
  $g=2\sin(3x_2)$ for $x_1<0$ and $g=\sin(3x_2)$ for $x_1>0$,
  drift vector $\bmu=0$. Plot for primal variable $u_h$: contour plot (left),
  surface plot (right).}
  \label{Results-101}
\end{figure}

Our last numerical experiment was conducted on a problem for which
the exact solution is not only discontinuous, but also not known
to us. This test problem has the following configuration: domain
$\Omega=(-1,1)\times (-1,1)$, diffusion tensor $a=\alpha I$ with
$\alpha=1$ in the first and third quadrant and $\alpha=10$ in the
second and fourth quadrant, load function $f=\frac14$, Dirichlet
boundary data $g=0$, drift vector $\bmu=[1,1]$. The numerical
solution for the primal variable can be seen in Figure
\ref{Results-102}. The left figure is the contour plot, and the
one on right is a 3D surface plot.

\begin{figure}[h]
\centering
\begin{tabular}{cc}
  \resizebox{2.4in}{1.8in}{\includegraphics{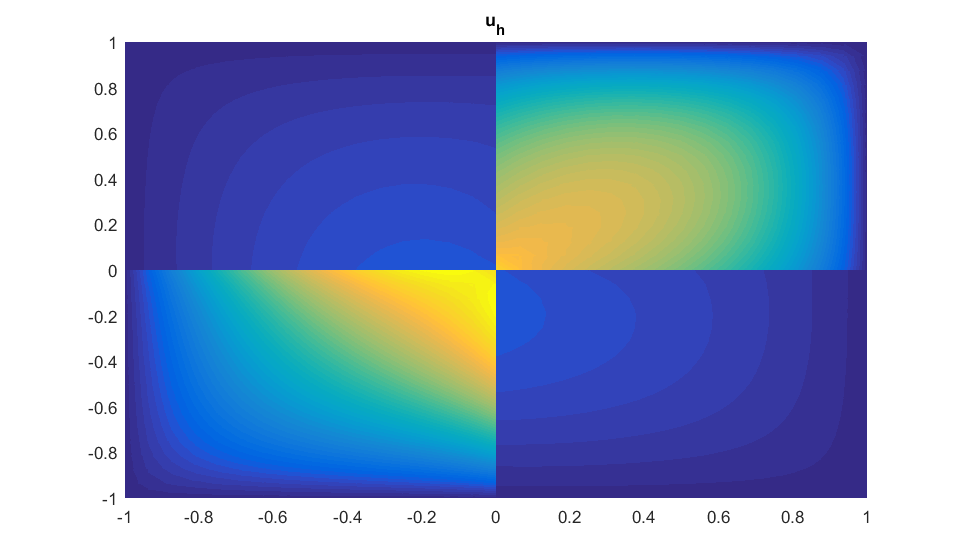}}
  \resizebox{2.4in}{2.1in}{\includegraphics{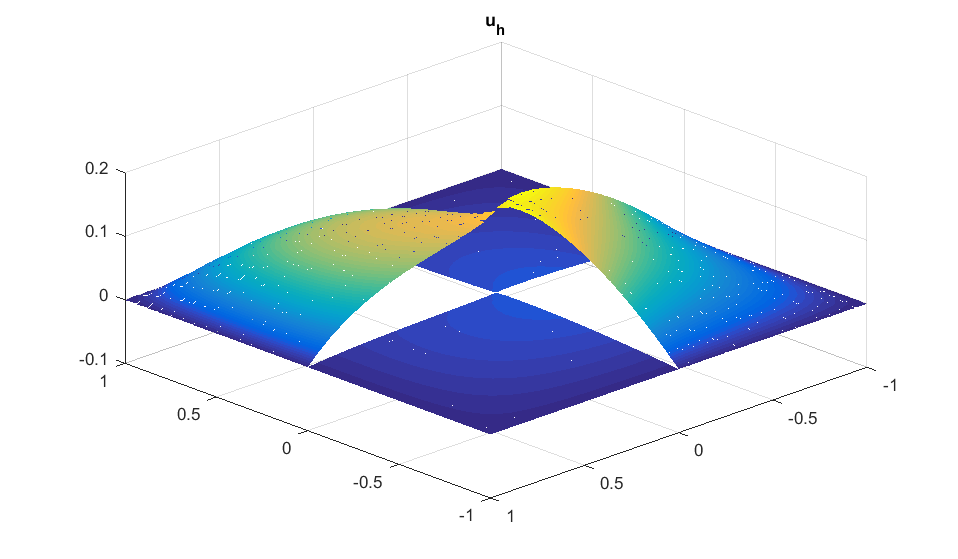}}
\end{tabular}
  \caption{Domain $\Omega=(-1,1)\times (-1,1)$, diffusion tensor $a=\alpha I$ with $\alpha=1$ in the first and third quadrant and $\alpha=10$ in the second and fourth
  quadrant, load function $f=\frac14$, Dirichlet boundary data $g=0$, drift vector $\bmu=[1,1]$. Plot for primal variable $u_h$: contour plot (left),
  surface plot (right).}
  \label{Results-102}
\end{figure}

\end{document}